\documentclass[a4paper]{article}

\usepackage{graphicx} % Required for inserting images
\usepackage{amsthm, amsmath, amssymb, mathtools, nicematrix, natbib, url, hyperref}
\usepackage{tikz-cd}
\usepackage{faktor}
\usepackage[margin=2.54cm]{geometry}

\setcitestyle{numbers}
\setcitestyle{square}

\numberwithin{equation}{section}

\newtheorem{theorem}{Theorem}[section]
\newtheorem{lemma}[theorem]{Lemma}
\newtheorem{proposition}[theorem]{Proposition}

\newtheorem{cor}[theorem]{Corollary}
\theoremstyle{remark}
\newtheorem{remark}[theorem]{Remark}

\title{$G_2$-Poisson equation on homogeneous spheres}
\author{Stepan Hudecek\thanks{School of Mathematics and Physics, The University of Queensland, St Lucia,~QLD 4072, Australia} \\
\small{\texttt{s.hudecek@uq.edu.au}}}
\date{}

\begin{document}

\maketitle

\begin{abstract}
    This paper studies the Poisson equation for the $G_2$-Laplacian on 3-forms on the 7-sphere that are invariant under a transitive group action. We establish the existence and uniqueness of $G$-invariant solutions for $G=SU(4),\: Spin(7),\: (Sp(2)\times Sp(1))/\mathbb{Z}_2$. In the case $G=Sp(2)\times U(1)/\mathbb{Z}_2$, we show that the operator does not preserve the set of positive 3-forms. The paper also discusses the eigenvalue problem for the $G_2$-Laplacian. We classify $G$-invariant solutions for the above choices of~$G$ and determine which of these solutions are nearly parallel $G_2$-structures.
\end{abstract}

\section{Introduction}

A choice of a smooth everywhere positive 3-form $\varphi$ on a 7-dimensional manifold is called a $G_2$-structure. We say that a compact manifold is a $G_2$-manifold if and only if it supports a $G_2$-structure $\varphi$ such that
\begin{equation}\label{G2-Laplace}
    \Delta_\varphi\varphi=0,
\end{equation}
where $\Delta_\varphi$ is the Hodge-Laplace operator associated to the metric induced by $\varphi$.
Some global solutions of~\eqref{G2-Laplace} are known (see, e.g.,~\cite{Joyce-book, nordstrom}), but producing a general theory of such solutions is rather difficult. No invariant solutions exist on homogeneous spaces, because all homogeneous $G_2$-manifolds are Ricci-flat and Ricci-flat homogeneous manifolds are flat. This impedes the analysis of~\eqref{G2-Laplace} substantially since homogeneous spaces typically provide a fruitful testing ground for geometric PDEs; see, e.g.,~\cite{BohmZiller, Bohm2018-cy}.

The problem of constructing $G_2$-manifolds has been approached from several different perspectives. One may study the heat flow of $G_2$-structures associated to the equation \eqref{G2-Laplace}, hoping to establish the convergence of its solutions to forms satisfying~\eqref{G2-Laplace} (see, for example, the survey \cite{Lotay2020}). Note that the flow has also been studied in a homogeneous setting~\cite{Jorge-flow-ofhomogeneous}. Alternatively, one may consider the co-flow, i.e., the flow of the associated 4-form $*\varphi$ (see ~\cite{Grigorian2020}). A weakly parabolic variant of this co-flow, the so-called modified co-flow depending on a parameter~$A$, was introduced in~\cite{GRIGORIAN2013378}. For certain choices of $A$, some of the critical points of the modified co-flow are nearly parallel $G_2$-structures (for the discussion of the parameter see \cite{BEDULLI2020107030}), i.e., $G_2$-structures $\varphi$ satisfying $d\varphi = \tau_0*_\varphi\varphi$ for some function~$\tau_0$. These $G_2$-structures solve the eigenvalue problem
\begin{equation}\label{eigenvalue problem for the G2-Laplacian}
    \Delta_\varphi\varphi = c^2\varphi.
\end{equation}
Equation~\eqref{eigenvalue problem for the G2-Laplacian}, as well as its generalisations, has arisen in~\cite{lotay2025nearlyg2manifoldsg2laplaciancoflows, Jorge-flow-ofhomogeneous, G2-eigenforms}.

In light of these developments, it is natural to investigate the properties of the $G_2$-Laplace operator
\begin{equation*}
    \varphi\mapsto\Delta_{\varphi}\varphi.
\end{equation*}
Particularly, one would like to understand the image of this operator, which amounts to studying the existence of solutions to the Poisson equation
\begin{equation} \label{Poisson equation - general}
    \Delta_\varphi\varphi=\mu.
\end{equation}
It was shown in \cite{Artem_Tim} that local solutions exist for every closed positive form $\mu$. Unlike the Laplace equation~\eqref{G2-Laplace}, the Poisson equation~\eqref{Poisson equation - general} admits homogeneous solutions for invariant right-hand side~$\mu$ in many interesting cases. This creates an opportunity to analyse the $G_2$-Laplace operator in the framework of homogeneous spaces.

Spheres provide a natural starting point for the study of geometric PDEs; cf.~\cite{Ziller1982,BPRZ}. For dimensional reasons, only the 7-sphere, $S^7$, admits $G_2$-structures. However, there are exactly seven groups with free and transitive actions on the 7-sphere, which yield multiple distinct homogeneous structures. These groups are 
\begin{equation*}
    SO(8),\:U(4),\:Spin(7),\:SU(4),\: \faktor{Sp(2)\times Sp(1)}{\mathbb{Z}_2}, \:\faktor{Sp(2)\times U(1)}{\mathbb{Z}_2},\:Sp(2).
\end{equation*}
\sloppy There are no $SO(8)$- or $U(4)$-invariant $G_2$-structures (see Section \ref{Section 4} below). In this paper, we explicitly describe the spaces of $G$-invariant $G_2$-structures on the 7-dimensional sphere for ${G=Spin(7)},{SU(4)},{(Sp(2)\times Sp(1))/\mathbb{Z}_2}$ and ${(Sp(2)\times U(1))/\mathbb{Z}_2}$. For the $Sp(2)$-invariant $G_2$-structures, see \cite{SP2-invariant}.
Building on this analysis, we establish the following result for three of the four cases we consider.
\begin{theorem}\label{theorem}
    Let $G=SU(4),\: Spin(7)$ or $(Sp(2)\times Sp(1))/\mathbb{Z}_2$. Then 
    the $G_2$-Laplacian is an orientation preserving isomorphism on the space of positive $G$-invariant 3-forms on the 7-sphere.
\end{theorem}

\sloppy As a corollary, we obtain the existence and uniqueness of invariant solutions to the \mbox{$G_2$-Poisson} equation for invariant right-hand sides for the aforementioned $G$.
In the remaining case of ${G=(Sp(2)\times U(1))/\mathbb{Z}_2}$, we prove a similar existence result in a neighbourhood of 3-forms with $(Sp(2)\times Sp(1))/\mathbb{Z}_2$-symmetry; see Theorem~\ref{thm for Sp(2)U(1)}. Furthermore, we are able to show that the image of the $G_2$-Laplacian does not consist of only positive 3-forms. This implies the existence of a global solution to the Poisson equation with everywhere degenerate right-hand side. Thus, a direct extension of Theorem~\ref{theorem} to this case is not possible. To further advance our understanding of the $(Sp(2)\times Sp(1))/\mathbb{Z}_2$-invariant case, we plot the image of the $G_2$-Laplacian with $(Sp(2)\times U(1))/\mathbb{Z}_2$ symmetry using \texttt{Matplotlib}. Moreover, we establish numerically the non-uniqueness of solutions to~\eqref{Poisson equation - general}. This demonstrates that the uniqueness portion of Theorem~\ref{theorem} fails for the $(Sp(2)\times U(1))/\mathbb{Z}_2$-symmetry as well.

Next, we classify the solutions to the eigenvalue problem~\eqref{eigenvalue problem for the G2-Laplacian} for each of the symmetries we consider. All of the solutions are coclosed forms, and for the symmetries $G=SU(4),\: Spin(7)$ and ${(Sp(2)\times Sp(1))/\mathbb{Z}_2}$ they are exactly the nearly parallel $G_2$-structures. For the $SU(4)$ symmetry, they form a continuous 1-dimensional family and for $Spin(7)$ and  $(Sp(2)\times Sp(1))/\mathbb{Z}_2$ they are a discrete subset up to scaling. However, for $(Sp(2)\times U(1))/\mathbb{Z}_2$ one obtains, up to scaling, a discrete subset of nearly parallel $G_2$-structures and a continuous family of eigenvectors that are coclosed but are not nearly parallel. 

The paper is organised as follows. Section~2 establishes preliminaries about $G_2$-structures and introduces our notation.
In Section~3, we describe spaces of $G$-invariant $G_2$-structures and prove a technical lemma about orthogonal bases for different Hodge stars.
Section~4 focuses on homogeneous structures of the sphere $S^7$ and decompositions of the corresponding isotropy representations.
Section~5 is devoted to the explicit constructions and analysis of $G$-invariant $G_2$-structures. In Propositions \ref{positive 3-forms for Spin(7)}, \ref{positive 3-forms for SU(4)}, \ref{positive 3-forms for Sp(2)U(1)} and \ref{positive 3-forms for Sp(2)Sp(1)} we characterise positive $G$-invariant 3-forms for each~$G$, and we prove Theorem \ref{theorem} as a consequence of Theorems \ref{thm for Spin(7)}, \ref{thm for SU(4)} and \ref{thm for Sp(2)Sp(1)}. The same Theorems together with \ref{thm:fixed points for Sp(2)Sp(1)} classify the solutions to the eigenvalue problems (see Remark \ref{relation between eigenvalue and fixed points} on how they relate). In the last section we describe the numerical experiments we have done.
\vspace{10pt}

\noindent\textbf{Acknowledgements.} The author would like to thank his advisors, Artem Pulemotov and Timothy Buttsworth, for their support during the project, and Kyle Broder and Ramiro Lafuente for the helpful discussions.

\section{Preliminaries}

We use the definition of $G_2$ given in \cite{Bryant}. For a general 3-form $\varphi\in\Lambda^3\mathbb{R}^7$ on a real 7-dimensional vector space, we consider the relation
\begin{equation}\label{defining equation for G2}
    \frac{1}{6}\iota_X\varphi\wedge\iota_Y\varphi\wedge\varphi=g(X,Y)\operatorname{dvol}_g,
\end{equation}
where $g$ is a symmetric bilinear form, $\iota$ is the interior product and 
$\operatorname{dvol}_g$ is a top form associated to the bilinear form $g$ and an orientation, that is, $\operatorname{dvol}_g = \pm\sqrt{|\det(g_{ij})|}e^1\wedge\ldots\wedge e^n$ for an orthonormal basis $e^1,\ldots, e^7$.
If there is non-zero $X\in\mathbb{R}^7$ such that the left-hand side vanishes for $Y=X$, we say that the 3-form is degenerate; otherwise, formula \eqref{defining equation for G2} defines a unique (non-degenerate) symmetric bilinear form $g\in S^2(\mathbb{R}^7)$ with signature $(7,0)$ or signature $(3,4)$. If it is positive definite, then $g$ is an inner product, denoted by $g_\varphi$, and we say that the 3-form $\varphi$ is positive.
Furthermore, stabilisers of positive 3-forms under the natural action of $GL(7)$ on $\mathbb{R}^7$ are conjugate to each other. We define the group $G_2$ to be the stabiliser of the 3-form
\begin{equation*}
    \varphi_0=e^{123}+e^{145}+e^{167}+e^{246}-e^{257}-e^{347}-e^{356},
\end{equation*}
where $e^{ijk}=e^i\wedge e^j\wedge e^k$. Since any positive 3-form $\varphi$ induces a metric and a volume form, $G_2\leq SO(7)$, we have a Hodge star operator $*_\varphi$ at our disposal.
\begin{remark}\label{bundle of G2 structures}
    There is a sequence of homomorphisms of Lie groups $G_2\to SO(7)\to SL(7)$
    that induces a fibre bundle
    \begin{equation*}
        \faktor{SO(7)}{G_2}\to \faktor{SL(7)}{G_2}\to \faktor{SL(7)}{SO(7)},
    \end{equation*}
    where $SO(7)/G_2$ is the space of all positive 3-forms inducing the same metric with a fixed volume form, $SL(7)/G_2$ is the space of all positive 3-forms inducing the same volume form and $SL(7)/SO(7)$ is the space of all metrics with a fixed volume form, respectively.
    Since the space of metrics is contractible, this is a trivial bundle
    \begin{equation*}
        \faktor{SL(7)}{G_2}\cong\faktor{SO(7)}{G_2}\times \faktor{SL(7)}{SO(7)}\cong\mathbb{R}P^7\times\faktor{SL(7)}{SO(7)}.
    \end{equation*}
    For a given subgroup $G\leq G_2$, we obtain a similar fibre bundle
    \begin{equation*}
        \Big(\faktor{SO(7)}{G_2}\Big)^G\to \Big(\faktor{SL(7)}{G_2}\Big)^G\to \Big(\faktor{SL(7)}{SO(7)}\Big)^G,
    \end{equation*}
    where the superscript $^G$ denotes only those tensors that are $G$-invariant. Again, this bundle is trivial. Altogether we obtain the decomposition for the space of all $G$-invariant $G_2$-structures
    \begin{equation}\label{decomposition of the space of G2-structures}
        \left(\faktor{GL(7)}{G_2}\right)^G\cong
        \mathbb{R}^*\times\left(\faktor{SL(7)}{G_2}\right)^G\cong
        \mathbb{Z}_2 \times \mathbb{R}^+ \times\left(\faktor{SL(7)}{SO(7)}\right)^G \times \left(\faktor{SO(7)}{G_2}\right)^G  ,
    \end{equation}
    where in the first isomorphism we took out the scaling factor and in the second we used the observations above and further decomposed the scaling into a positive scaling and a sign. The factors correspond to choosing an orientation, a volume form (with the same orientation), a metric (with the fixed volume form) and $G_2$-structure (giving the same metric).
\end{remark}
We call a smooth everywhere positive 3-form on a manifold a $G_2$-structure. A manifold admits a $G_2$-structure if and only if it is orientable and spinnable, i.e., the first two Stiefel-Whitney classes vanish~\cite{Gray-G_2}.
In this case, the space of $G_2$-structures maps onto the space of $SO(7)$-structures (metrics and orientations).

\sloppy On a manifold with $G_2$-structure $\varphi$, one may define the Hodge Laplace operator by ${\Delta_\varphi = (d+\delta_\varphi)^2}$, where $d$ is the exterior derivative and $\delta_\varphi$ is its dual, the codifferential, defined on $k$-forms by ${\delta_\varphi=(-1)^k*_\varphi d*_\varphi}$.
In particular, for 3-forms we obtain $\Delta_\varphi = *_\varphi d*_\varphi d-d*_\varphi d*_\varphi$.
We refer to the map $\varphi\mapsto\Delta_\varphi\varphi$ as the $G_2$-Laplacian; it is not a linear operator due to its dependence on the $G_2$-structure. We include the proof of its scaling properties for the convenience of the reader.

\vspace{2mm}

\begin{lemma}\label{scaling of G2 Laplacian}
The $G_2$-Laplacian satisfies
\begin{equation*}
    \Delta_{c\varphi}c\varphi = c^{\frac{1}{3}}\Delta_\varphi \varphi
\end{equation*}
for $c\in\mathbb{R}\backslash\{0\}$.
\end{lemma}

\begin{proof}
The metric is induced via equation \eqref{defining equation for G2}.
If we now scale $\varphi$ to $c\varphi$, we obtain
\begin{equation*}
    g_{c\varphi}(X,Y) \text{dvol}_{g_{c\varphi}} = 
    \frac{1}{6}\iota_X c\varphi\wedge\iota_Y c\varphi\wedge c\varphi = \frac{1}{6}c^3\iota_X\varphi\wedge\iota_Y\varphi\wedge\varphi = c^3 g_\varphi(X,Y) \text{dvol}_{g_\varphi}.
\end{equation*}
Note that $\text{dvol}_{cg} = c^{\frac{7}{2}}\text{dvol}_g$, which implies that $g_{c\varphi} = c^{\frac{2}{3}}g_\varphi$.
Combining with the formula for Hodge star $*_{cg} = c^{\frac{7-2k}{2}}*_g$, we find that on $k$-forms
\begin{equation*}
    *_{g_{c\varphi}} = *_{c^\frac{2}{3}g_\varphi} = c^{\frac{2}{3}\frac{7-2k}{2}}*_{g_\varphi}= c^\frac{7-2k}{3}*_{g_\varphi}.
\end{equation*}
For the $G_2$-Laplacian,
\begin{equation*}
    \Delta_{c\varphi}c\varphi = (*_{c\varphi}d*_{c\varphi}d - d*_{c\varphi}d*_{c\varphi})c\varphi = (c^{-\frac{1}{3}}c^{-\frac{1}{3}}*_\varphi d*_\varphi d - c^{-1} c^\frac{1}{3}d*_\varphi d*_\varphi)c\varphi = c^\frac{1}{3}\Delta_\varphi\varphi.
\end{equation*}
\end{proof}

\begin{remark}\label{eigenvalue problem - rmk about lambda}
    Even though the $G_2$-Laplacian is non-linear, we can consider the eigenvalue problem
    \begin{equation*}
        \Delta_\varphi\varphi = \lambda\varphi.
    \end{equation*}
    Of course, because of the scaling properties, the absolute value of $\lambda$ is of no interest; however its sign and the eigenform $\varphi$ itself are. 
    On a compact manifold, a solution $\varphi$ to the eigenvalue problem can only ever exist for a non-negative $\lambda$ since $\Delta_\varphi$ is a self-adjoint positive semi-definite operator (with respect to the metric induced by $\varphi$).
\end{remark}

\begin{remark}\label{fixed points}
    Lemma \ref{scaling of G2 Laplacian} immediately implies that the $G_2$-Laplacian preserves lines passing through the origin, i.e., it can be considered as an operator on the projectivisation $\Delta_-(-): \mathbb{P}\Omega^3(M) \to \mathbb{P}\Omega^3(M)$. One implication of this is that, whenever there is a fixed point for the $G_2$-Laplacian, there is a line that is mapped onto itself, i.e., a fixed point in the projective space. Conversely, for each line that is mapped onto itself there are exactly two fixed points, one for each orientation.
\end{remark}

We say that a compact manifold is a $G_2$-manifold if it supports a $G_2$-structure $\varphi$ such that $\Delta_\varphi\varphi=0$ (see \cite{Bryant}). These manifolds have their holonomy inside $G_2$.

\section{Homogeneous $G_2$-structures}
For a homogeneous space $G/H$, where $G$ is a compact Lie group and $H$ is its closed subgroup, there is an $Ad_H$-invariant decomposition at the level of Lie algebrae, i.e., $\mathfrak{g} = \mathfrak{h} \oplus \mathfrak{m}$.
The space of $G$-invariant 3-forms on the manifold is then isomorphic to the space of $Ad_H$-invariant 3-forms on the vector space $\mathfrak{m}$, i.e., $\Omega^3(G/H)^G \cong (\Lambda^3\mathfrak{m}^*)^H$. It follows that a homogeneous manifold admits a $G$-invariant $G_2$-structure if and only if $Ad(H)\leq G_2$, where we interpret $Ad$ as a map from $H$ to $GL(\mathfrak{m})$.

Since the space of possible $G_2$-structures is open in the space of 3-forms, it follows that if $G/H$ admits a $G$-invariant $G_2$-structure, the dimension of the space of positive 3-forms is the same as the dimension of the space of all $H$-invariant 3-forms.
Consider the decomposition of $\Lambda^3(\mathbb{R}^7)$ under the natural action of $G_2\leq GL(7,\mathbb{R})$ into irreducible representations $\Lambda^3(\mathbb{R}^7) = \Lambda^3_{27}(\mathbb{R}^7) \oplus \Lambda^3_7(\mathbb{R}^7) \oplus \Lambda^3_1(\mathbb{R}^7)$. As was shown in~\cite{bryant2025remarksg2structures}, there are $G_2$-equivariant isomorphisms
\begin{align*}
    \Lambda^3_{27}(\mathbb{R}^7) &\cong S^2_0(\mathbb{R}^7),\\
    \Lambda^3_7(\mathbb{R}^7) &\cong \Lambda^1(\mathbb{R}^7),
\end{align*}
where $S^2_0$ denotes the space of traceless symmetric 2-tensors. From these isomorphisms, we can see that the space of $G$-invariant $G_2$-structures on $G/H$ is
\begin{equation} \label{number of G2 structures}
\begin{split}
    \Omega^3\Big(\faktor{G}{H}\Big)^G &\cong \Lambda^3(\mathfrak{m}^*)^H \cong \Lambda^3_{27}(\mathfrak{m}^*)^H \oplus \Lambda^3_7(\mathfrak{m}^*)^H \oplus \Lambda^3_1(\mathfrak{m}^*)^H \cong\\
    &\cong S^2_0(\mathfrak{m}^*)^H \oplus \Lambda^1(\mathfrak{m}^*)^H \oplus \Lambda^3_1(\mathfrak{m}^*)^H.
\end{split}
\end{equation}
Therefore, the dimension of the space of $G$-invariant 3-forms on $G/H$ is the same as the dimension of the space of all $H$-invariant metrics plus the dimension of the space of $H$-invariant vector fields on $\mathfrak{m}$. In a special case where there are no equivalent summands in the isotropy representation and none of them is a trivial summand, the dimension of the space of all possible $H$-invariant $G_2$-structures is the number of irreducible summands. And if there is a trivial summand, it is the number of irreducible summands plus one.

\begin{remark}
  The condition on $G/H$ to admit a $G_2$-structure is essential. Consider a homogeneous 7-sphere $S^7\cong U(4)/U(3)$. It admits $U(3)$-invariant 3-forms since there is a $U(3)$-invariant 1-form (trivial component in the isotropy representation) and a K\"ahler form (because of the standard representation of $U(3)$), but there are no $G_2$-structures since $U(3)$ is not a subgroup of $G_2$.  
\end{remark}
There is a similar decomposition (see \cite{bryant2025remarksg2structures}) for the 2-forms under the action of $G_2$,
\begin{equation*}
    \Lambda^2(\mathbb{R}^7) = \Lambda^2_7(\mathbb{R}^7)\oplus\Lambda^2_{14}(\mathbb{R}^7),
\end{equation*}
where the factor $\Lambda^2_{14}(\mathbb{R}^7)$ corresponds to the adjoint representation of the group $G_2$ on its Lie algebra $\mathfrak{g}_2$ and $\Lambda^2_7(\mathbb{R}^7)\cong\Lambda^1(\mathbb{R}^7)$, as in the case of 3-forms. Passing to $H$-invariant forms, we obtain
\begin{equation}\label{number of 2-forms}
    \Omega^2\Big(\faktor{G}{H}\Big)^G \cong \Lambda^2(\mathfrak{m}^*)^H\cong\Lambda^1(\mathfrak{m}^*)^H\oplus(\mathfrak{g}_2)^H.
\end{equation}
Therefore, the dimension of the space of $G$-invariant 2-forms on $G/H$ is the same as the dimension of the space of all $H$-invariant vector fields plus the number of trivial summands in the decomposition of the Lie algebra $\mathfrak{g}_2$ considered as the representation of $H\leq G_2$.

\vspace{2mm}

We will need the following technical lemma.

\begin{lemma} \label{Invariant basis for * - lemma}
    Let $\mathfrak{m} = \mathfrak{m}_1 \oplus \ldots \oplus \mathfrak{m}_k$ be a decomposition of a representation of a Lie group $H$ into irreducible representations $\mathfrak{m}_i$. If there are no equivalent summands, then for every $l\in \mathbb{N}$ there exist bases $\beta_l$ and $\beta^*_l$ of $\Lambda^l(\mathfrak{m})^H$ and $\Lambda^{n-l}(\mathfrak{m})^H$ respectively such that for each $H$-invariant inner product on $\mathfrak{m}$ the associated Hodge star $*:\Lambda^l(\mathfrak{m})^H \to \Lambda^{n-l}(\mathfrak{m})^H$ has diagonal matrix with respect to these bases.
\end{lemma}

\begin{proof}
    Consider the decomposition    
    \begin{equation}\label{decomposition}
        \Lambda^l(\mathfrak{m})= \bigoplus_{l_1+\ldots +l_k=k}\Lambda^{l_1}(\mathfrak{m}_1)\wedge\ldots\wedge \Lambda^{l_k}(\mathfrak{m}_k) =: \bigoplus_{j=1}^p W_j.
    \end{equation}
    Note that $W_j$ are $H$-invariant subspaces for all $j$, i.e., $H\cdot W_j=W_j$. Take $w \in \Lambda^l(\mathfrak{m})^H\subseteq\Lambda^l(\mathfrak{m})$ and write it as $w=w_1+\ldots +w_p$, where $w_j\in W_j$ as in the decomposition \eqref{decomposition}. Note that 
    \begin{equation*}
        w_1+\ldots +w_p = w= h\cdot w = h\cdot w_1 +  \ldots + h\cdot w_p,
    \end{equation*}
    but since $h\cdot w_j\in H\cdot W_j = W_j$, we necessarily have $h\cdot w_j = w_j$ because the decomposition is unique. Thus we have shown that
    \begin{equation*}
        \Lambda^l(\mathfrak{m})^H= \bigoplus_{l_1+\ldots l_k=k}\big(\Lambda^{l_1}(\mathfrak{m}_1)\wedge\ldots\wedge \Lambda^{l_k}(\mathfrak{m}_k)\big)^H = \bigoplus_{j=1}^p W_j^H.
    \end{equation*}
    Let us now fix $H$-invariant inner products $g_i$ on each irreducible component $\mathfrak{m}_i$. By Schur's lemma, they are unique up to scaling.
    Using Schur's lemma again and the inequivalence of the summands, we find that all $H$-invariant inner products on $\mathfrak{m}$ are of the form $g=a_1 g_1 \oplus \ldots \oplus a_k g_k$, where $a_i$ are positive real numbers.
    This implies that any orthogonal basis with respect to one of the $H$-invariant metrics is orthogonal with respect to every $H$-invariant metric.
    
    We now fix two $H$-invariant metrics on $\mathfrak{m}$: background metric $g=g_1 \oplus \ldots \oplus g_k$ and any general metric $h=a_1 g_1 \oplus \ldots \oplus a_k g_k$.
    We choose an orientation (i.e., a top form $\operatorname{dvol}$) of $\mathfrak{m}$ and we will write the Hodge star corresponding to this orientation and the metric $g$ as $*_g$.
    We also choose an orientation of each irreducible summand compatible with the one on $\mathfrak{m}$, that is, $\operatorname{dvol_0}\wedge\ldots\wedge \operatorname{dvol}_k = \operatorname{dvol}$ (the orientations are not uniquely determined by this relation). 
    With these orientations we get Hodge stars $*_{g_i}$ on each of the summands.
    Take arbitrary $w=w_1\wedge\ldots\wedge w_k$ and $v=v_1\wedge\ldots\wedge v_k$ such that each $w_i,v_i$ lie in $\Lambda^{l_i}\mathfrak{m}_i$ (some $l_i$ may be 0). Consider the defining equation for the Hodge star
    \begin{equation*}
    *_gw\wedge v = \Tilde{g}(w,v) \operatorname{dvol} = \Tilde{g_1}(w_1,v_1)\ldots \Tilde{g_k}(w_k,v_k)\operatorname{dvol_1}\wedge\ldots\wedge\operatorname{dvol}_k = *_{g_1}w_1\wedge v_1\wedge \ldots \wedge*_{g_k}w_k\wedge v_k,
    \end{equation*}
    where $\Tilde{g}$ and $\Tilde{g_i}$ are the induced metrics on the corresponding spaces of forms.
    From the above equality, we obtain the up-to sign relation
    \begin{equation*}
        *_g(w_1\wedge\ldots\wedge w_k) = \pm*_{g_1}w_1\wedge\ldots\wedge*_{g_k}w_k.
    \end{equation*}
    It is easy to check that when we take an element $w\in\Lambda^l\mathfrak{m}_i\subseteq \Lambda^l\mathfrak{m}$, we obtain $*_{cg_i}w = c^{\frac{n_i-2l}{2}}*_{g_i}w$, where $n_i$ is the dimension of $\mathfrak{m}_i$.    
    It follows that for an element $w\in\Lambda^{l_1}\mathfrak{m}_1\wedge,\ldots,\wedge \Lambda^{l_k}\mathfrak{m}_k$, where $l_1+\ldots+l_k = l$, we obtain 
    \begin{equation}\label{Hodge star with coefficients}
        *_h w = \pm a_1^\frac{n_1-2l_1}{2}\cdot\ldots\cdot a_k^\frac{n_k-2l_k}{2}*_g w.
    \end{equation}
    Therefore, on each $W_j$ the operators $*_h$ and $*_g$ differ only by a constant.
    
    We can now choose an orthonormal basis of $W_j^H$ for every $j\in\{1,\ldots,p\}$ with respect to the background metric $g$.
    Joining these bases creates a basis $\beta$. We define $\beta^*:= *_g \beta$.
    It is now clear that the matrix of $*_g$ with respect to the bases $\beta$ and $\beta^*$ is the identity matrix .
    Because of the previous observation, $*_h$ has a diagonal matrix with respect to these bases as well.
\end{proof}

\begin{remark}\label{wedging invariant forms is not enough}
The spaces $W_j^H$ defined above are orthogonal to each other for each $l$ regardless of the $H$-invariant inner product.    
\end{remark}

\begin{remark}
In the proof above, $\Lambda^l(\mathfrak{m})^H$ is not necessarily equal to
    \begin{equation*}
        \bigoplus_{l_1+\ldots l_k=k} \Lambda^{l_1}(\mathfrak{m}_1)^H\wedge\ldots\wedge \Lambda^{l_k}(\mathfrak{m}_k)^H.
    \end{equation*}
    Consider, for example, the action of $O(2)$ on $\mathbb{R}^2\oplus\mathbb{R}$ given by the standard and the determinant representation, respectively. Clearly, $\Lambda^2(\mathbb{R}^2)^{O(2)}$ is zero because there is a reflection in the group, but also $\Lambda^1(\mathbb{R})^{O(2)}$ is zero because the reflection has determinant $-1$. However, these effects can cancel out, that is, $\Lambda^3(\mathbb{R}^2\oplus\mathbb{R})^{O(2)}=(\Lambda^2(\mathbb{R}^2)\wedge \Lambda^1(\mathbb{R}))^{O(2)} = span\{ e_1\wedge e_2 \wedge e_3\}$ is non-zero.
\end{remark}

\begin{cor} \label{Basis with resp star - corollary}
    Let $G/H$ be a compact homogeneous space of dimension $n$ with a reductive decomposition $\mathfrak{g}=\mathfrak{h} \oplus \mathfrak{m}$, where $\mathfrak{m}$ is an $H$-invariant complement of $\mathfrak{h}$. If there are no equivalent summands in $\mathfrak{m}$ considered as a representation of $H$, then for every $l$, there exist bases of $\Omega^l\big(G/H\big)^G$ and $\Omega^{n-l}\big(G/H\big)^G$ with respect to which any $G$-invariant Hodge star has a diagonal matrix.
\end{cor}

\begin{proof}
    The corollary is a straightforward consequence of the lemma and the isomorphism $\Omega^l(G/H)^G \cong \Lambda^l(\mathfrak{m^*})^H$.
\end{proof}

\section{Invariant $G_2$-structures on the sphere $S^7$}\label{Section 4}
First, note that $S^7$ is orientable and spinnable since necessarily the first two Stiefel-Whitney classes vanish ($w_1(S^7) = w_2(S^7) = 0$) because $H^1(S^7;\mathbb{Z}_2) = H^2(S^7;\mathbb{Z}_2) = 0$. Therefore, every metric comes from some $G_2$-structure (see, for example, \cite{bryant2025remarksg2structures}). We will look at those $G_2$-structures that are invariant under some transitive group action.
There are only seven Lie groups acting transitively and effectively on the $S^7$; see \cite{Oniscik}, \cite{Montgomery1943TransformationGO} and~\cite{Borel-sphere}. They can be thought of inside a poset diagram of subgroups of $SO(8)$, as seen in Figure 1.

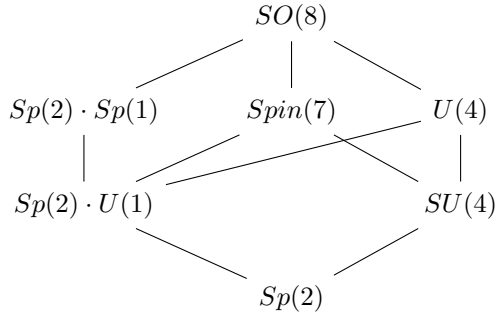
\begin{figure}[h]
\centering
    \begin{tikzcd}
& SO(8) \arrow[rd, no head] \arrow[d, no head] \arrow[ld, no head] &                           \\
Sp(2)\cdot Sp(1) \arrow[d, no head]                      & Spin(7) \arrow[rd, no head] \arrow[ld, no head]                  & U(4) \arrow[d, no head]   \\
Sp(2)\cdot U(1) \arrow[rd, no head] \arrow[rru, no head] &                             & SU(4) \arrow[ld, no head] \\
& Sp(2)                                                            &
\end{tikzcd}
\caption{Poset diagram of subgroups of $SO(8)$ acting transitively on $S^7$}
\end{figure}
\noindent Here the dot $\cdot$ represents taking a product and factoring by $\mathbb{Z}_2$, e.g., $Sp(2)\cdot Sp(1) =\allowbreak (Sp(2)\times Sp(1))/\mathbb{Z}_2$.
The nontrivial inclusions are consequences of low-dimensional isomorphisms, that is,
\begin{align*}
    Sp(2)\cdot U(1) \cong Spin(5)\cdot Spin(2) \leq Spin(7) &&
    \text{and} &&
    SU(4)\cong Spin(6)\leq Spin(7).
\end{align*}
We obtain possible presentations of $S^7$ as a homogeneous space $G/H$ where $G$ and $H$ are from the following table:

\begin{table}[h!]
\centering
\begin{tabular}{ c|c }
        $G$ & $H$\\
        \hline
        $SO(8)$ & $SO(7)$\\
        $U(4)$ & $U(3)$\\
        $Spin(7)$ & $G_2$\\
        $SU(4)$ & $SU(3)$\\
        $Sp(2)\cdot Sp(1)$ & $Sp(1)\cdot Sp(1)$\\
        $Sp(2)\cdot U(1)$ & $Sp(1)\cdot U(1)$\\
        $Sp(2)$ & $Sp(1)$
\end{tabular}
\caption{Presentations of $S^7$ as a homogeneous space}
\end{table}

There are no $SO(8)$- and $U(4)$-invariant $G_2$-structures on the 7-sphere since the corresponding isotropy groups, that is, $SO(7)$ and $U(3)$, are not subgroups of $G_2$. (For the list of all compact manifolds that admit $G_2$-structures, consult \cite{REIDEGELD} and \cite{HongVanLee}).
If we look at the isotropy decompositions of the homogeneous spheres admitting $G_2$-structures, we get the following (see \cite{Ziller1982}):
\begin{equation}\label{isotropy decompostions}
\begin{split}
    \mathfrak{spin}(7) &= \mathfrak{g_2} \oplus \mathfrak{m}_7,\\
    \mathfrak{su}(4) &= \mathfrak{su}(3) \oplus \mathfrak{m}_1 \oplus \mathfrak{m}_6,\\
    \mathfrak{sp}(2)\oplus\mathfrak{sp}(1) &= \mathfrak{sp}(1)\oplus\mathfrak{sp}(1) \oplus \mathfrak{m}_4 \oplus \mathfrak{m}_3,\\
    \mathfrak{sp}(2)\oplus\mathfrak{u}(1) &= \mathfrak{sp}(1)\oplus\mathfrak{u}(1) \oplus\mathfrak{m}_4\oplus\mathfrak{m}_2\oplus\mathfrak{m}_1,\\
    \mathfrak{sp}(2)&=\mathfrak{sp}(1)\oplus\mathfrak{m}_4\oplus\mathfrak{m}_1\oplus\mathfrak{m}_1\oplus\mathfrak{m}_1.    
\end{split}
\end{equation}
\noindent
Here, the indices denote the real dimensions of the irreducible summands. We can see that for all but the last space we can use Corollary \ref{Basis with resp star - corollary} above.
On the other hand, all the one-dimensional summands in the last decomposition are trivial summands and hence equivalent.

Consider also the decompositions of the Lie algebra $\mathfrak{g}_2$ under the restrictions of the adjoint representations to these subgroups:
\begin{equation}
    \begin{split}
        \mathfrak{g}_2^{SU(3)}&=\mathfrak{su}(3)\oplus\mathfrak{m}_6,\\
        \mathfrak{g}_2^{Sp(1)\cdot Sp(1)}&=\mathfrak{sp}(1)\oplus\mathfrak{sp}(1)\oplus\mathfrak{m_8},\\
        \mathfrak{g}_2^{Sp(1)\cdot U(1)}&=\mathfrak{sp}(1)\oplus\mathfrak{u}(1)\oplus\mathfrak{m}_2\oplus\mathfrak{m}_4\oplus\mathfrak{m}_4,\\
        \mathfrak{g}_2^{Sp(1)}&=\mathfrak{sp}(1)\oplus\mathfrak{m}_1\oplus\mathfrak{m}_1\oplus\mathfrak{m}_1\oplus\mathfrak{m}_4\oplus\mathfrak{m}_4.\\
    \end{split}
\end{equation}

\vspace{2mm}
Using the observations below equations \eqref{number of G2 structures} and \eqref{number of 2-forms} as well as both of the decompositions above, we see that the dimensions of the spaces of invariant forms of degree at most 3 are:

\begin{table}[h!]
\centering
\begin{tabular}{ c|c|c|c }
        $G$ & $\operatorname{dim}\big(\Omega^1(S^7)^G)$ & $\operatorname{dim}\big(\Omega^2(S^7)^G )$&           $\operatorname{dim}\big(\Omega^3(S^7)^G)$ \\
        \hline
        $Spin(7)$ & 0 & 0 & 1 \\
        $SU(4)$ & 1 & 1 & 3\\
        $Sp(2)\cdot Sp(1)$ & 0 & 0 & 2\\
        $Sp(2)\cdot U(1)$ & 1 & 2 & 4 \\
        $Sp(2)$ & 3 & 6 & 10
\end{tabular}
\caption{Dimensions of the spaces of invariant forms on $S^7$}
\label{number of inv. k-forms for spheres}
\end{table}
Of course, the dimensions of the other spaces are easily obtainable via the isomorphisms $\Omega^k(S^7)^G\cong \Omega^{7-k}(S^7)^G$ for each $k\in \{1,2,3\}$.

\section{Solving the Poisson equation}

For a homogeneous manifold, the space of invariant 3-forms $\Omega^3(G/H)^G\cong\Lambda^3(\mathfrak{g}/\mathfrak{h})^H$ is finite-dimensional. For a fixed metric, the Hodge Laplacian is a linear operator on this space. Consequently, after fixing a basis, we can interpret the $G_2$-Laplacian as a matrix with entries parametrised by the input $\varphi$. The process for solving the Poisson equation for each symmetry type is the following.
\begin{itemize}
    \item[] \textbf{Step 1}. Take an $Ad_H$-invariant complement $\mathfrak{m}$ to the Lie algebra $\mathfrak{h}$ so that $\mathfrak{g}=\mathfrak{h}\oplus\mathfrak{m}$ and choose a suitable basis for the complement $\mathfrak{m}$ and bases for spaces of invariant forms as in Corollary \ref{Basis with resp star - corollary}. 
    \item[] \textbf{Step 2}. Using \eqref{defining equation for G2}, determine the conditions that ensure the positivity of 3-forms (i.e., that the symmetric bilinear form associated to the 3-form is positive definite). These forms are precisely the $G_2$-structures on the 7-sphere with particular symmetry.
    \item[] \textbf{Step 3}. Compute the matrix of the $G_2$-Laplace operator using formula \eqref{Hodge star with coefficients} for the Hodge star and the following well-known formula for the exterior derivative on a homogeneous space:
    \begin{equation}\label{differential in homogeneus space}
        d\alpha(X_1,\ldots,X_{k+1}) = \sum_{i<j}(-1)^{i+j}\alpha([X_i,X_j]_\mathfrak{m},\ldots, \hat{X_i},\ldots,\hat{X_j},\ldots,X_{k+1}),
    \end{equation}
    where $[-,-]_\mathfrak{m}$ is the projection of the commutator to the subspace $\mathfrak{m}\subseteq \mathfrak{g}$.
    Because of our choice of basis, the matrix for the Hodge star will be a diagonal matrix, albeit dependent on the metric.
    \item[] \textbf{Step 4}. Finally, solve the Poisson equation and the eigenvalue problem using the resulting expression. Both of these are now equations with rational functions on the left-hand sides.
\end{itemize}

\begin{remark}\label{relation between eigenvalue and fixed points}
    For the eigenvalue problem
    \begin{equation*}
        \Delta_\varphi\varphi=\lambda\varphi,
    \end{equation*}
    we have noted in Remark \ref{eigenvalue problem - rmk about lambda} that $\lambda$ is non-negative. It also cannot be zero for a non-vanishing $\varphi$ since $H^3(S^7;\mathbb{Z})=0$. Therefore, all of the eigenforms are positive multiples of the fixed points.
\end{remark}

We will now carry out these steps for each of the symmetry types in Theorem \ref{theorem}. We are using conventions that $\mathbb{Z}_2=\{-1,1\}$ and $\mathbb{R}^+ = (0,\infty)$.

\subsection{$Spin(7)$-symmetry}
We begin with the simplest case of $Spin(7)$-symmetry.
As seen in equation \eqref{isotropy decompostions}, the isotropy representation is an irreducible representation of the group $G_2$. In fact, it is the standard representation that comes from the inclusion $G_2\leq SO(7)$ and corresponds to the automorphisms of the octonions.

We fix a $G_2$-invariant 3-form $\varphi_0$ and the induced $G_2$-invariant metric $g_0$. From Table \ref{number of inv. k-forms for spheres} and the fact that the $G_2$ representation is irreducible, they are both unique up to scaling.
Trivially, we obtain the following proposition (cf. Remark \ref{bundle of G2 structures} and equation \eqref{decomposition of the space of G2-structures}).

\begin{proposition}\label{positive 3-forms for Spin(7)}
The map taking $(\sigma,a)\in \mathbb{Z}_2\times \mathbb{R}^+$ to the 3-form
    \begin{equation*}
        \varphi = A^3\varphi_0,
    \end{equation*}
where $A=\sigma a$, is a diffeomorphism onto the space of $Spin(7)$-invariant positive 3-forms on the 7-sphere. The metric induced by the $G_2$-structure $\varphi$ is then given by
    \begin{equation*}
        g_\varphi = A^2g_0.
    \end{equation*}
\end{proposition}

\begin{remark}
    This is an example of a metric with more symmetries than the given 3-form. The metric $g_\varphi$ is actually $SO(8)$-invariant (it is a round metric), while the 3-form $\varphi$ inducing it is just $Spin(7)$-invariant. Indeed, the 3-form cannot be $SO(8)$-invariant as the isotropy group of this action, $SO(7)$, is transitive on orthonormal triples.
\end{remark}

For a basis of the space $\Lambda^4(\mathfrak{m}_7)^{G_2}$, we choose $*_{g_0}\varphi_0$. Obviously, any $G_2$-invariant 3-form is coclosed with respect to any metric because $\Lambda^5(\mathfrak{m}_7)^{G_2}\cong\Lambda^2(\mathfrak{m}_7)^{G_2}=\{0\}$ (see Table \ref{number of inv. k-forms for spheres}). Therefore, the associated Laplace operator is just $\Delta_\varphi = *_\varphi d*_\varphi d$. The operator $d$ with respect to the bases $\{\varphi_0\}$ and $\{*\varphi_0\}$ is multiplication by some $K\in \mathbb{R}$. The $G_2$-Laplacian then becomes
\begin{equation*}
    \varphi_0\xmapsto{d}K*_{g_0}\varphi_0\xmapsto{*_{g_0}} K\varphi_0 \xmapsto{d}K^2*_{g_0}\varphi_0\xmapsto{*_{g_0}} K^2\varphi_0.
\end{equation*}
Using Lemma \ref{scaling of G2 Laplacian} we obtain
\begin{equation*}
    \Delta_{A^3\varphi}A^3\varphi_0 = A\Delta_\varphi\varphi_0 = AK^2\varphi_0.
\end{equation*}
In other words, the $G_2$-Laplacian acts (up to a constant multiple) like the cube root function. Consequently, we have the following Theorem.

\begin{theorem}\label{thm for Spin(7)}
    The $G_2$-Laplace operator $\varphi\mapsto \Delta_\varphi\varphi$ is an orientation-preserving bijection on the space of positive $Spin(7)$-invariant 3-forms on the 7-sphere and has exactly two fixed points which are $\pm K^3\varphi_0$.
\end{theorem}\noindent
\begin{remark}
    The constant $K$ can be made explicit by choosing a basis for $\mathfrak{m}_7$ and computing the differential on 3-forms.
\end{remark}
There is an immediate corollary.
\begin{cor}
    For every positive $Spin(7)$-invariant 3-form $\mu$ there is a unique positive $Spin(7)$-invariant solution $\varphi$ to the Poisson equation
    \begin{equation*}
        \Delta_\varphi\varphi=\mu.
    \end{equation*}
\end{cor}

\vspace{1mm}

\subsection{$SU(4)$-symmetry}
First, let us recall the reductive decomposition for $\mathfrak{su}(4)$ from \eqref{isotropy decompostions}, which can be expressed as follows:

\begin{equation*}
    \mathfrak{su}(4) = \mathfrak{m}_1 \oplus \mathfrak{m}_6 \oplus \mathfrak{su}(3)=\mathbb{R}\begin{pNiceArray}{c|c}
  \boldsymbol{i} & 0 \\
  \hline
  0 & -\boldsymbol{i}I_3
\end{pNiceArray}\oplus \begin{pNiceArray}{c|c} 
  0 & -(\mathbb{C}^3)^* \\
  \hline
  \mathbb{C}^3 & 0
\end{pNiceArray} \oplus \begin{pNiceArray}{c|c}
  0 & 0 \\
  \hline
  0 & \mathfrak{su}(3)
\end{pNiceArray},
\end{equation*}
where $\mathbb{C}^3$ is the standard representation of $SU(3)$ and $\mathbb{R}$ is the trivial representation.

Now, we choose a basis for $\mathfrak{m} = \mathfrak{m}_1\oplus \mathfrak{m}_6$. Take the standard complex basis of $\mathbb{C}^3$ and denote it by $\{\xi_j\}$ for $j\in \{1,2,3\}$, and define

\begin{align} \label{basis for m in SU(3)}
    e_1=\begin{pNiceArray}{c|c}
  3\boldsymbol{i} & 0 \\
  \hline
  0 & -\boldsymbol{i}I_3
\end{pNiceArray} & &
e_{2j}=\begin{pNiceArray}{c|c}
  0 & -\xi_j^T \\
  \hline
  \xi_j & 0
\end{pNiceArray} & &
e_{2j+1}=\begin{pNiceArray}{c|c}
  0 & \boldsymbol{i}\xi_j^T \\
  \hline
  \boldsymbol{i}\xi_j & 0
\end{pNiceArray}, & &
j\in\{1,2,3\}.
\end{align}
We will denote with upper indices the elements of the dual basis.
There are four indecomposable $SU(3)$-invariant forms of degree at most 3 (cf. Table \ref{number of inv. k-forms for spheres}):
\begin{align}\label{primitive SU(3) invariant forms}
    e^1, &&
    \omega=e^{23}+e^{45}+e^{67}, &&
    V_R=e^{246}-e^{257}-e^{347}-e^{356}, &&
    V_I=e^{346}-e^{357}+e^{247}+e^{256}.
\end{align}
Here we mean indecomposable in the sense that they cannot be split into a wedge product of $SU(3)$-invariant forms.
These correspond to the well-known K\"ahler form and to the real and imaginary part of the complex volume form on $\mathbb{C}^3$.
Consequently, the set $\beta =\{e^1\wedge \omega, V_R,V_I\}$ is a basis of the space of invariant 3-forms satisfying the condition in Corollary \ref{Basis with resp star - corollary}.
We are going to compute the metric induced by a general 3-form to find the coefficients for the Hodge star operator.
For reasons that will become clear soon, let us choose the following parametrisation of the space of 3-forms:
\begin{equation} \label{parametrisation of SU(3)-inv forms}
    \varphi=A^3(e^1\wedge\omega)+R^3\cos(\alpha)V_R+ R^3\sin(\alpha)V_I,
\end{equation}
where $A\in \mathbb{R}$, $R\in \mathbb{R}^+\cup\{0\}$ and $\alpha\in[0,2\pi)$.
The relation for the metric \eqref{defining equation for G2} yields
\begin{equation*}
    g_\varphi(e_i,e_j) \operatorname{dvol}_{\varphi}= \frac{1}{6}\iota_{e_i}\varphi\wedge\iota_{e_j}\varphi\wedge\varphi =\begin{cases}
        A^9e^{1234567}, &\text{if } i=j=1,\\
        A^3R^6e^{1234567}, &\text{if } i=j\neq 1,\\
        0, &\text{otherwise}.
    \end{cases}
\end{equation*}
Since we have chosen the basis \eqref{basis for m in SU(3)} above, we know that the metric must be diagonal (see Lemma \ref{Invariant basis for * - lemma}). One can consider the above equation as an equation of matrices (indices $i,j$) with 7-forms as entries. By computing the determinant of the coefficients on both sides (see \cite{Karigiannis-G2andSpin7}), we obtain
\begin{equation*}
    \det\Big ( \sqrt{|\det g_{ij}|}\: g_{ij} \Big)=|\det g_{ij}|^\frac{7}{2}(\det g_{ij}) = A^{27} R^{36}.
\end{equation*}
We immediately see that the condition for the form to be non-degenerate is $A\ne 0$ and $R \ne 0$.
From here, we obtain the unique volume form and the metric induced by this 3-form:
\begin{align*}
    \operatorname{dvol}_{\varphi} = A^3 R^4 e^{1234567}, &&
    g_\varphi = A^6R^{-4} e^1\otimes e^1 + R^2 \sum_{i=2}^7 e^i\otimes e^i.
\end{align*}
\begin{remark}
    As expected, we can clearly see that if the $3$-form scales by $\lambda^3$, then the metric is scaled by a factor $\lambda^2$ and the volume form scales by $\lambda^7$.
\end{remark}
We gather the observations made above in the following proposition (we use slightly different parametrisation than suggested in Remark \ref{bundle of G2 structures}).
\begin{proposition}\label{positive 3-forms for SU(4)}
The map taking $(\sigma, (a, R), \alpha)\in \mathbb{Z}_2\times (\mathbb{R}^+)^2\times S^1$ to the 3-form
    \begin{equation*}
        \varphi=A^3(e^1\wedge\omega)+R^3\cos(\alpha)V_R+ R^3\sin(\alpha)V_I,
    \end{equation*}
where $A=\sigma a$ and $e^1\wedge\omega,V_R$ and $V_I$ are as in \eqref{primitive SU(3) invariant forms}, is a diffeomorphism onto the space of $SU(4)$-invariant positive 3-forms on the 7-sphere. The metric induced by the $G_2$-structure $\varphi$ is given by
    \begin{equation*}
        g_\varphi = A^6R^{-4} e^1\otimes e^1 + R^2 \sum_{i=2}^7 e^i\otimes e^i.
    \end{equation*}
\end{proposition}
\begin{remark}
    It follows that there are only positive 3-forms and degenerate 3-forms, that is, there are no 3-forms inducing a metric with signature $(4,3)$. This is not surprising since $SU(3)$ is contained in $G_2$ but not in the stabiliser of the non-degenerate 3-forms with indefinite signature.
\end{remark}

Next, we are going to compute the Hodge star and the differential.
Denote by $*$ the Hodge star corresponding to the metric with $A=R=1$ and $\alpha=0$ (and $\operatorname{dvol} = e^{1234567}$). The space of $SU(3)$-invariant 4-forms is
\begin{equation*}
    \Lambda^4(\mathfrak{m})^{SU(3)}=span\{*(e^1\wedge\omega), *V_R, *V_I\}.
\end{equation*}
Let us denote by $*\beta$ the basis $\{*(e^1\wedge\omega),*V_R,*V_I\}$. 
As noted in the proof of Lemma \ref{Invariant basis for * - lemma}, the Hodge star is given by formula \eqref{Hodge star with coefficients}, which in this case translates into
\begin{align*}
    *_{g_\varphi} (e^1\wedge\omega) &= A^{-3}R^2R^2*(e^1\wedge\omega), \\
    *_{g_\varphi}V_R&= A^3R^{-2}*V_R, \\
    *_{g_\varphi}V_I&= A^3R^{-2}*V_I.
\end{align*}
Therefore, the matrix of the Hodge star $*_{g\varphi}$ from 3-forms to 4-forms with respect to the bases $\beta$ and $*\beta$ is
\begin{equation*}
    \begin{pNiceArray}{c c}
  A^{-3}R^4 & 0 \\
  0 & A^{3}R^{-2}I_2
    \end{pNiceArray}.
\end{equation*}
To compute the exterior differentials, one needs to compute the projections of the Lie brackets under the above decomposition. An easy calculation gives
\begin{align*}
    [e_{2j},e_{2j+1}]_\mathfrak{m} = -\frac{2}{3}e_1, &&
    [e_1, e_{2j}]_\mathfrak{m} = [e_1, e_{2j}] = -4e_{2j+1}, &&
    [e_1, e_{2j+1}]_\mathfrak{m} = [e_1, e_{2j+1}] = 4e_{2j}, &&
    j \in \{1,2,3\}.
\end{align*}
Using formula \eqref{differential in homogeneus space}, we get 
\begin{align} \label{diffenrentials for SU(3)}
    d(e^1\wedge\omega) = \frac{2}{3}\omega\wedge\omega = \frac{4}{3}*(e^1\wedge\omega), &&
    dV_R=12*V_R, &&
    dV_I = 12*V_I.
\end{align}
Since the differential is a bijection between invariant 3-forms and 4-forms, every 4-form is closed, and hence also every 3-form is coclosed (with respect to any invariant metric).
Because of that, the Laplace operator on the $SU(4)$-invariant 3-forms induced by a positive 3-form $\varphi$ as in \eqref{parametrisation of SU(3)-inv forms} is given by (with respect to the basis $\beta$) 
\begin{equation*}
    \Delta_\varphi = *_\varphi d*_\varphi d = \begin{pNiceArray}{c c}
  \frac{16}{3}A^{6}R^{-8} & 0 \\
  0 & 144A^{-6}R^4I_2
    \end{pNiceArray}.
\end{equation*}
The $G_2$-Laplacian is then
\begin{equation}\label{SU(3)-invariant G_2-Laplacian}
    \Delta_\varphi\varphi =  \frac{16}{3}A^{9}R^{-8} (e^1\wedge\omega) + 144A^{-6}R^7(\cos(\alpha)V_R+\sin(\alpha)V_I).
\end{equation}
Observe that the Laplace operator $\Delta_\varphi$ does not depend on the parameter $\alpha$. Furthermore, we obtain the following useful lemma.

\begin{lemma}
    With the coordinates $(\alpha,A,R)$ introduced above, the $G_2$-Laplacian commutes with arbitrary rotations around the $A$-axis.
\end{lemma}

\begin{proof}
    Let us denote by $T_\theta$ the rotation by an angle $\theta$, i.e.,
    \begin{align*}
        T_\theta(\varphi)&=T_\theta(A^3(e^1\wedge\omega)+R^3\cos(\alpha)V_R+ R^3\sin(\alpha)V_I)\\
        &= A^3(e^1\wedge\omega)+R^3\cos(\alpha+\theta)V_R+ R^3\sin(\alpha+\theta)V_I.
    \end{align*}
    We have observed that the Laplace operator is independent of the angle $\alpha$ of the 3-form $\varphi$, that is, $\Delta_{T_\theta(\varphi)}=\Delta_\varphi$.
    In addition, it acts by scaling in the last two coordinates, and therefore it commutes with any linear transformation there. We conclude the proof by the computation
    \begin{equation*}
        \Delta_{T_\theta(\varphi)}(T_\theta(\varphi)) = (\Delta_\varphi\circ T_\theta)(\varphi) = (T_\theta\circ\Delta_\varphi)(\varphi).
    \end{equation*}
    The first part of the above equation is just a consequence of the fact that $(SO(7)/G_2)^{SU(3)}$ is a circle parametrised by $\alpha$.
    The second equality follows from the explicit form of the Laplacian.
\end{proof}

We now state the main result of this subsection.
\begin{theorem}\label{thm for SU(4)}
    The $G_2$-Laplace operator is an orientation-preserving bijection on $SU(4)$-invariant positive 3-forms on the 7-sphere and has exactly two circles of fixed points.
\end{theorem}
\begin{proof}
    Because of the lemma above, it is enough to show that the $G_2$-Laplacian is a bijection on the space of positive 3-forms with $\alpha=0$. This subspace is clearly preserved.
    From the explicit expression of the $G_2$-Laplacian \eqref{SU(3)-invariant G_2-Laplacian}, substituting $\alpha=0$, we obtain the operator
    \begin{equation*}
        \Delta_{(-)}(-)|_{\alpha=0}:(A,R)\mapsto \Big(\frac{16}{3}A^9R^{-8},144A^{-6}R^7\Big).
    \end{equation*}
    The 3-form is positive if and only if both of the coordinates $(A,R)$ are non-zero, and the orientation is given by the sign of the parameter $A$.
    Clearly, applying $\Delta_{(-)}(-)$ to a positive 3-form yields a positive 3-form, and in addition, the first parameter $\frac{16}{3}A^9R^{-8}$ has the same sign as $A$, hence the orientation is preserved.
    The inverse of the $G_2$-Laplacian is given by
    \begin{equation*}
        (X,Y)\mapsto \bigg(\bigg(\bigg(\frac{3X}{16}\bigg)^7\bigg(\frac{Y}{144}\bigg)^8\bigg)^\frac{1}{15}, \bigg(\bigg(\bigg(\frac{3X}{16}\bigg)^2\bigg(\frac{Y}{144}\bigg)^3\bigg)^\frac{1}{5}\bigg).
\end{equation*}
Regarding fixed points, it is enough to realise that the line generated by $(\sqrt[14]{27},1)$ is preserved and then use Remark \ref{fixed points} and the above lemma.
\end{proof}\noindent
\begin{cor}
    For every positive $SU(4)$-invariant 3-form $\mu$, there is a unique positive $SU(4)$-invariant solution $\varphi$ to the Poisson equation
    \begin{equation*}
        \Delta_\varphi\varphi=\mu.
    \end{equation*}
\end{cor}

\vspace{1mm}

\subsection{$Sp(2)\cdot Sp(1)$- and $Sp(2)\cdot U(1)$-symmetry}\label{Sp(2)U(1) subsection}

For the symmetry $Sp(2)\cdot Sp(1)$, we can consider the reductive decomposition
\begin{align*}
    \mathfrak{sp}(2) \oplus \mathfrak{sp}(1) &= \mathfrak{m}_4\oplus\mathfrak{m}_3\oplus\mathfrak{sp}(1)\oplus \mathfrak{sp}(1) =\\
    &=\begin{pNiceArray}{ccc}
  0 & 0 & 0 \\
  0 & 0 & -\mathbb{H}^*\\
  0 & \mathbb{H} & 0
\end{pNiceArray} \oplus \begin{pNiceArray}{ccc}
  0 & 0 & 0 \\
  0 & 0 & 0\\
  0 & 0 & \operatorname{Im}(\mathbb{H})
\end{pNiceArray} \oplus \begin{pNiceArray}{ccc}
  0 & 0 & 0 \\
  0 & \mathfrak{sp}(1) & 0\\
  0 & 0 & 0
\end{pNiceArray} \oplus \begin{pNiceArray}{ccc}
  \mathfrak{sp}(1) & 0 & 0 \\
  0 & 0 & 0\\
  0 & 0 & \mathfrak{sp}(1)
\end{pNiceArray},
\end{align*}
where all the matrix entries are quaternions.
Observe that the second $\mathfrak{sp}(1)$ is embedded diagonally into $\mathfrak{sp}(2)\oplus\mathfrak{sp}(1)$ (otherwise the action would not be free).

\begin{remark}
    The isotropy representation of $(Sp(2)\cdot Sp(1))/(Sp(1)\cdot Sp(1))$ is given by two double covers. Firstly, $Sp(1)\cdot Sp(1)$ double covers $SO(\mathbb{H})\cong SO(4)$ via the map $(q_1,q_2)\mapsto (x\mapsto q_1xq_2^{-1})$, and secondly, $Sp(1)$ double covers $SO(\operatorname{Im}(\mathbb{H}))\cong SO(3)$ via the map $q\mapsto (x\mapsto qxq^{-1})$, where in both maps the multiplication is in quaternions. The space $\mathfrak{m}_4$ (resp. $\mathfrak{m}_3$) is the standard representations of $SO(4)$ (resp. $SO(3)$) via the first (resp. second) double cover.
\end{remark}

We have a similar reductive decomposition in the case with the symmetry $Sp(2)\cdot U(1)$, except we further decompose $\operatorname{Im}(\mathbb{H})= \boldsymbol{i}\mathbb{R} \oplus \boldsymbol{j}\mathbb{R} \oplus \boldsymbol{k}\mathbb{R}$ and obtain
\begin{align*}
&\mathfrak{sp}(2)\oplus \mathfrak{u}(1) = \mathfrak{m}_4\oplus\mathfrak{m}_2\oplus\mathfrak{m}_1\oplus\mathfrak{sp}(1)\oplus \mathfrak{u}(1)=\\
&= \begin{pNiceArray}{ccc}
  0 & 0 & 0 \\
  0 & 0 & -\mathbb{H}^*\\
  0 & \mathbb{H} & 0
\end{pNiceArray}\oplus \begin{pNiceArray}{ccc}
  0 & 0 & 0 \\
  0 & 0 & 0\\
  0 & 0 & \boldsymbol{j}\mathbb{R}\oplus \boldsymbol{k}\mathbb{R}
  \end{pNiceArray}\oplus \begin{pNiceArray}{ccc}
  0 & 0 & 0 \\
  0 & 0 & 0\\
  0 & 0 & \boldsymbol{i}\mathbb{R}
\end{pNiceArray}\oplus \begin{pNiceArray}{ccc}
  0 & 0 & 0 \\
  0 & \mathfrak{sp}(1) & 0\\
  0 & 0 & 0
\end{pNiceArray}\oplus \begin{pNiceArray}{ccc}
  \boldsymbol{i}\mathbb{R} & 0 & 0 \\
  0 & 0 & 0\\
  0 & 0 & \boldsymbol{i}\mathbb{R}
\end{pNiceArray}.
\end{align*}
\begin{remark}
    The isotropy representation of $(Sp(2)\cdot U(1))/(Sp(1)\cdot U(1))$ is the same as in the case of $Sp(2)\cdot Sp(1)$, except that we act on $\text{Im}(\mathbb{H})$ only by $U(1)\leq Sp(1)$. Because of that, the second summand in the case $Sp(2)\cdot Sp(1)$, $\mathfrak{m}_3$, splits further into a trivial one ($\mathfrak{m}_1$) and a 2-dimensional one ($\mathfrak{m}_2$), where the action is given by the trivial representation and by rotation ($SO(2)\cong U(1)$), respectively.
\end{remark}

We can choose $\mathfrak{m}$ to be
\begin{equation*}
    \mathfrak{m} = \begin{pNiceArray}{ccc}
  0 & 0 & 0 \\
  0 & 0 & -\mathbb{H}^*\\
  0 & \mathbb{H} & 0
\end{pNiceArray} \oplus \begin{pNiceArray}{ccc}
  0 & 0 & 0 \\
  0 & 0 & 0\\
  0 & 0 & \operatorname{Im}(\mathbb{H})
\end{pNiceArray}.
\end{equation*}
This $\mathfrak{m}$ is a subspace of $\mathfrak{sp}(2)\leq \mathfrak{sp}(2)\oplus\mathfrak{u}(1)\leq\mathfrak{sp}(2)\oplus\mathfrak{sp}(1)$; however, it is $Sp(1)\cdot Sp(1)$-invariant, with respect to the adjoint representation. Because of this, it can serve as the complement of the Lie subalgebra of the isotropy subgroup in both of the cases considered in this subsection. We choose the basis to be
\begin{align*}
    e_1 =\begin{pNiceArray}{ccc}
  0 & 0 & 0 \\
  0 & 0 & 0\\
  0 & 0 & \boldsymbol{i}
\end{pNiceArray}, &&
e_{2} = \begin{pNiceArray}{ccc}
  0 & 0 & 0 \\
  0 & 0 & 0\\
  0 & 0 & \boldsymbol{j}
\end{pNiceArray}, &&
e_{3} = \begin{pNiceArray}{ccc}
  0 & 0 & 0 \\
  0 & 0 & 0\\
  0 & 0 & -\boldsymbol{k}
\end{pNiceArray}, &&
e_{3+l} = \begin{pNiceArray}{ccc}
  0 & 0 & 0 \\
  0 & 0 & -\bar{\zeta_l}\\
  0 & \zeta_l & 0
\end{pNiceArray},
\end{align*}
where by $\zeta_l$ for $l\in\{1,2,3,4\}$, we mean the real basis $\{1,\boldsymbol{i},\boldsymbol{j},\boldsymbol{k}\}$ of the space of quaternions and the bar is the conjugation on $\mathbb{H}$.

\begin{remark}\label{Sp(2)U(1) is subgroup of SU(3)}
The reason we are choosing the third basis vector ($e_3$) to have an entry $-\boldsymbol{k}$ instead of $\boldsymbol{k}$ is so that the image of the embedding of $Sp(1)\cdot U(1)$ into $SO(7)$ via the adjoint representation is in $G_2$ as we have defined it.
In other words, if we chose a basis with entries corresponding to the set $\{\boldsymbol{i},\boldsymbol{j},\boldsymbol{k},1,\boldsymbol{i},\boldsymbol{j},\boldsymbol{k}\}$, we would not stabilise the standard 3-form $\varphi_0$ and the image of the adjoint representation would be in a conjugate subgroup to $G_2$.
Also, with this choice of basis, we have $Sp(1)\cdot U(1)\leq SU(3)\leq G_2$. Such an inclusion is, of course, not true for the subgroup $Sp(1)\cdot Sp(1)\cong SO(4)\leq G_2$ which is, in fact, a maximal subgroup of $G_2$.
\end{remark}
From now on, we will work only with the space $(Sp(2)\cdot U(1))/(Sp(1)\cdot U(1))$ and at the end of the computation we will reflect on the parts where there would be necessary changes for the bigger group. Observe that every $Sp(2)\cdot Sp(1)$-invariant tensor will also be $Sp(2)\cdot U(1)$-invariant.

The $Sp(2)\cdot U(1)$-invariant indecomposable forms of degree at most $3$ on $S^7$ are (cf. Table \ref{number of inv. k-forms for spheres}): 
\begin{gather}
    \begin{align*}    
    e^1, &&
    e^{23}, &&
    \omega_1 = e^{45}+e^{67}, &&
    V_R = e^2\wedge\omega_2+e^3\wedge\omega_3=e^{246}-e^{257}-e^{347}-e^{356},
    \end{align*} \\
    V_I = e^3\wedge\omega_2-e^2\wedge\omega_3=e^{346}-e^{357}+e^{247}+e^{256}.
    \label{primitive Sp(1)U(1) invariant forms}
\end{gather}
The form $e^{23}$ is a volume element on $\mathfrak{m}_2$, the form $\omega_1$ is a K\"ahler form on $\mathfrak{m}_4\cong\mathbb{C}^2$, and the last two are the real and imaginary part of the complex volume form on $\mathfrak{m}_4\oplus\mathfrak{m}_2\cong \mathbb{C}^3$.
Therefore,
\begin{equation}\label{basis for Sp(1)U(1) 3-forms}
    \{e^1\wedge e^{23}, e^1\wedge\omega_1, V_R, V_I\}
\end{equation}
is a basis for the space of invariant 3-forms. Again, this basis satisfies the conclusion of Corollary \ref{Basis with resp star - corollary}, and we will denote it $\beta$. Let us choose the parametrisation of the space of $Sp(2)\cdot U(1)$-invariant 3-forms by 
\begin{equation}\label{parametrisation for Sp(2)U(1) invariant 3-forms}
    \varphi = A^3e^1\wedge e^{23} + B^3e^1\wedge\omega_1 + R^3\cos(\alpha)V_R+R^3\sin(\alpha)V_I,
\end{equation}
where $A,B \in \mathbb{R}$, $R\in \mathbb{R}^+\cup\{0\}$ and $\alpha\in [0,2\pi)$. From \eqref{defining equation for G2} we obtain the equation
\begin{align*}
    g_\varphi(e_i,e_j) \operatorname{dvol}_\varphi = \begin{cases}
        A^3B^6e^{1234567}, &\text{if } i=j=1,\\
        A^3R^6e^{1234567}, &\text{if } i=j\in\{2,3\},\\
        B^3R^6e^{1234567}, &\text{if } i=j\in\{4,5,6,7\},\\
        0, &\text{otherwise}.
    \end{cases}
\end{align*}
Again, it was expected that we would obtain a diagonal metric because of our choice of basis. We compute the determinant
\begin{equation*}
    |\det g_{ij}|^\frac{7}{2}(\det g_{ij}) = A^9B^{18}R^{36}.
\end{equation*}
Therefore, we obtain the conditions for non-degeneracy as $A,B\neq 0$ and $R\neq  0$. From the determinant we immediately arrive at
\begin{align*}
    \operatorname{dvol}_\varphi=AB^2R^4e^{1234567}, &&
    g_\varphi = A^2B^4R^{-4}(e^1\otimes e^1) + A^2B^{-2}R^2\sum_{i=2}^{3}e^i\otimes e^i + A^{-1}BR^2\sum_{i=4}^{7}e^i\otimes e^i.
\end{align*}
From here we can read off the condition for the positivity of $\varphi$, and it is that $A$ and $B$ must have the same sign; otherwise, the signature of the metric would be $(3,4)$. We have proved the following proposition (cf. Remark \ref{decomposition of the space of G2-structures})

\begin{proposition}\label{positive 3-forms for Sp(2)U(1)}
The map taking $(\sigma, (a,b, R), \alpha)\in \mathbb{Z}_2\times (\mathbb{R}^+)^3\times S^1$ to the 3-form
    \begin{equation*}
        \varphi=A^3e^1\wedge e^{23} + B^3e^1\wedge\omega_1 + R^3\cos(\alpha)V_R+R^3\sin(\alpha)V_I
    \end{equation*}
where $A=\sigma a$, $B=\sigma b$ and $e^1\wedge e^{23}, e^1\wedge\omega,V_R$ and $V_I$ are as in \eqref{primitive Sp(1)U(1) invariant forms}, is a diffeomorphism onto the space of $Sp(2)\cdot U(1)$-invariant positive 3-forms on the 7-sphere. The metric induced by the $G_2$-structure $\varphi$ is given by
    \begin{equation*}
        g_\varphi = A^2B^4R^{-4}(e^1\otimes e^1) + A^2B^{-2}R^2\sum_{i=2}^{3}e^i\otimes e^i + A^{-1}BR^2\sum_{i=4}^{7}e^i\otimes e^i.
    \end{equation*}
\end{proposition}

Next, we are going to compute the Hodge star $*_\varphi$ with respect to the basis \eqref{basis for Sp(1)U(1) 3-forms} for invariant 3-forms and the basis $\{*(e^1\wedge e^{23}), *(e^1\wedge\omega_1), *V_R, *V_I\}$, where we denote by $*$ the Hodge star with respect to the 3-form given by $A=B=R=1$ and $\alpha=0$.
Formula \eqref{Hodge star with coefficients} yields
\begin{align*}
    *_{g_\varphi}(e^1\wedge e^{23}) &= A^{-1}B^{-2}R^2 \cdot A^{-2}B^2R^{-2} \cdot A^{-2}B^2R^4*(e^1\wedge e^{23}) = A^{-5}B^{2}R^{4}*(e^1\wedge e^{23}),\\
    *_{g_\varphi}(e^1\wedge\omega_1) &= A^{-1}B^{-2}R^2 \cdot A^2B^{-2}R^2 \cdot 1*(e^1\wedge\omega_1) = AB^{-4}R^4*(e^1\wedge\omega_1),\\
    *_{g_\varphi}V_R&=AB^2R^{-2} \cdot 1 \cdot 1 * V_R = AB^2R^{-2}*V_R,\\
    *_{g_\varphi}V_I&=AB^2R^{-2} \cdot 1 \cdot 1 * V_I = AB^2R^{-2}*V_I.
\end{align*}
We will also need the Hodge star $*_\varphi$ on 2-forms, where the same formula gives
\begin{align*}
    *_\varphi\omega_1 &= (AB^2R^{-1} \cdot A^2B^{-2}R^2 \cdot 1)*\omega_1 = A^3*\omega_1,\\
    *_\varphi e^{23} &= (AB^2R^{-2} \cdot A^{-2}B^2R^{-2} \cdot A^{-2}B^2R^4)*e^{23} = A^{-3}B^6*e^{23}.
\end{align*}
We compute the commutators now. In this case, there are many more commutators than before. For a more concise description of the commutators, we are going to use the notation 
\begin{align*}
    e_{i,0}=e_1, &&
    e_{j,0}=e_2, &&
    e_{-k,0}=e_3, &&
    e_{0,\zeta_l}=e_{3+l}.
\end{align*}
Furthermore, we consider the indices linear, i.e., by $e_{2i,0}$ we mean $2e_{i,0}$.
With this notation the projections of the commutators onto $\mathfrak{m}$ are
\begin{align*}
    [e_{\alpha,0},e_{\beta,0}]_\mathfrak{m} = [e_{\alpha,0},e_{\beta,0}] = e_{[\alpha,\beta],0}, &&
    [e_{\alpha,0},e_{0,\beta}]_\mathfrak{m} = [e_{\alpha,0},e_{0,\beta}] = e_{0,\alpha\cdot\beta},
\end{align*}
$$[e_{0,\alpha},e_{0,\beta}]_\mathfrak{m} = e_{-\alpha\cdot\bar{\beta}+\bar{\alpha}\cdot\beta,0}=\begin{cases}
        0, &\text{if } \alpha=\beta,\\
        e_{-2\alpha\bar{\beta},0}, &\text{otherwise},
\end{cases}$$
where $\alpha,\beta\in\mathbb{H}$ and the operations $[\cdot,\cdot]$ and $\cdot$ in the lower indices are the operations on the quaternions.
From the commutators we are now able to compute the differentials of the indecomposable invariant forms of degree at most 3 using \eqref{differential in homogeneus space}:
\begin{align*}
    de^1 = 2e^{23}-2\omega_1, &&
    d\omega_1 = -2V_I, &&
    de^{23}=-2V_I, &&
    dV_R = -4*(e^1\wedge\omega_1) -8*(e^{123}),&&
    dV_I = 0.
\end{align*}
Using these formulae as well as the relations 
\begin{align*}
    *(e^1\wedge e^{23})=e^{23}\wedge\omega_1, &&
    *(e^1\wedge e^{23})=\frac{1}{2}\omega_1^{\wedge2}, &&
    *V_R=-e^1\wedge V_I,
\end{align*}
we obtain the differentials of all invariant 3-forms:
\begin{align*}
    d(e^1\wedge e^{23})&= -2*(e^1\wedge \omega_1) - 2*V_R, &
    dV_R&= -4*(e^1\wedge\omega_1) - 8*(e^1\wedge e^{23}),\\
    d(e^1\wedge\omega_1)&= 2*(e^1\wedge\omega_1) - 4*(e^1\wedge e^{23}) - 2*V_R, &
    dV_I&=0.
\end{align*}
\begin{remark}\label{coclosed form are defined regardless of the metric}
    Because $d$ is an isomorphism from $span\{e^1\wedge e^{23}, e^1\wedge\omega_1, V_R\}$ to $span\{{*(e^1\wedge e^{23})}, {*(e^1\wedge\omega_1)}, *V_R\}$, and because all invariant Hodge stars have diagonal matrices for our choice of a basis (see Corollary \ref{Basis with resp star - corollary}), the 3-form is coclosed with respect to one invariant metric if and only if it is coclosed with respect to every invariant metric.
\end{remark}
We will also need that
\begin{equation*}
    d*V_I=d(e^1\wedge V_R)=2(e^{23}-\omega_1)\wedge V_R - 4e^1\wedge e^{23}\wedge V_R - 4e^1\wedge\omega_1^{\wedge2} = -4*\omega_1-8*e^{23},
\end{equation*}
where we have used that 
\begin{align*}
    *V_I=e^1\wedge V_R, &&
    2(e^{23}-\omega_1)\wedge V_R=0, &&
    4e^1\wedge e^{23}\wedge V_R=*\omega_1, &&
    e^1\wedge\omega_1^{\wedge2} = 2*e^{23}.
\end{align*}

\vspace{1mm}

We are now ready to compute the $G_2$-Laplace operator.
First, note that all the basis elements of $\beta$ are either closed ($V_I$) or coclosed ($e^{123},e^1\wedge\omega_1,V_R$), and the Laplacian on closed (resp. coclosed) forms reduces to $\Delta=d\delta$ (resp. to $\Delta = \delta d$). Let us fix $\varphi$ as in \eqref{parametrisation for Sp(2)U(1) invariant 3-forms}.
We will compute the Laplacian given by this 3-form on the closed 3-form $V_I$. We find
\begin{align*}
    \Delta_\varphi V_I &= -d*_\varphi d*_\varphi V_I = -d*_\varphi d(AB^2R^{-2}*V_I) = -AB^2R^{-2}d *_\varphi (-4*\omega_1-8*e^{23})\\
    &= -AB^2R^{-2}d(-4A^{-3}\omega_1-8A^3B^{-6}e^{23}) = -AB^2R^{-2}(8A^{-3} + 16A^3B^{-6})V_I\\
    &= -(8A^{-2}B^2 + 16A^4B^{-4})R^{-2}V_I.
\end{align*}
On the other hand, when we look at the operator restricted to the coclosed forms (coclosedness is well-defined regardless of the metric; see Remark \ref{coclosed form are defined regardless of the metric}), with respect to the basis $\{e^1\wedge e^{23},e^1\wedge\omega_1,V_R\}$ we find
\begin{equation}\label{G2-Laplacian on the coclosed part}
\begin{aligned}
    \Delta_\varphi|_{\operatorname{Ker}\delta_\varphi} &= *_\varphi d *_\varphi d =
    \left(
    \begin{pNiceArray}{ccc}
  A^5B^{-2}R^{-4} & 0 & 0 \\
  0 & A^{-1}B^4R^{-4} & 0\\
  0 & 0 & A^{-1}B^{-2}R^2
\end{pNiceArray}
\begin{pNiceArray}{ccc}
  0 & -4 & -8 \\
  -2 & 2 & -4\\
  -2 & -2 & 0   
\end{pNiceArray}\right)^2 =\\
&=\begin{pNiceArray}{ccc}
  8C_1C_2 + 16C_1C_3 & -8C_1C_2 + 16C_1C_3 & 16C_1C_2 \\
  -4C_2^2 + 8C_2C_3 & 8C_1C_2 + 4C_2^2 + 8C_2C_3 & 16C_1C_2 - 8C_2^2\\
  4C_2C_3 & +8C_1C_3 - 4C_2C_3 & 16C_1C_3 + 8C_2C_3
\end{pNiceArray},
\end{aligned}
\end{equation}
where $C_i$ represent the elements on the diagonal of the first matrix. That is,
\begin{align*}
    C_1C_2 &= A^4B^2R^{-8}, &
    C_1C_3 &= A^4B^{-4}R^{-2}, &
    C_2C_3 &= A^{-2}B^2R^{-2}, &
    C_2^2 &= A^{-2}B^8R^{-8}.
\end{align*}
Putting everything together, for a general 3-form given by \eqref{parametrisation for Sp(2)U(1) invariant 3-forms} and a general
\begin{equation*}
    \mu = Xe^1\wedge e^{23} + Ye^1\wedge\omega_1 + Q\cos(\beta)V_R+Q\sin(\beta)V_I,
\end{equation*}
the equation $\Delta_\varphi\varphi=\mu$ is equivalent to the system of equations
\begin{equation}\label{Poisson eq for Sp(2)U(1)}
\begin{split}
    8A^7B^2R^{-8} + 16A^7B^{-4}R^{-2} - 8A^4B^5R^{-8} + 16A^4B^{-1}R^{-2} + 16A^4B^2R^{-5}\cos(\alpha) &= X,\\
    - 4AB^8R^{-8} + 8AB^2R^{-2} + 8A^4B^5R^{-8} + 4A^{-2}B^{11}R^{-8} + 8A^{-2}B^5R^{-2} +\\
    + 16A^4B^2R^{-5}\cos(\alpha) - 8A^{-2}B^8R^{-5}\cos(\alpha) &= Y,\\
    4AB^2R^{-2} + 8A^4B^{-1}R^{-2} - 4A^{-2}B^5R^{-2} + 16A^4B^{-4}R\cos(\alpha) + 8A^{-2}B^2R\cos(\alpha) &= Q\cos(\beta),\\
    -8A^{-2}B^2R\sin(\alpha) - 16A^4B^{-4}R\sin(\alpha) &= Q\sin(\beta).
\end{split}
\end{equation}
We first discuss the case of $Sp(2)\cdot Sp(1)$. With the same basis \eqref{basis for Sp(1)U(1) 3-forms} there are only two $Sp(2)\cdot Sp(1)$-invariant 3-forms $e^{123}$ and $\sum_{i=1}^3 e^i\wedge\omega_1$. This implies that all of the computations above work exactly the same, except we have extra conditions $B^3=R^3\cos(\alpha)$ and $R^3\sin(\alpha)=0$.
We obtain the following result.

\begin{proposition}\label{positive 3-forms for Sp(2)Sp(1)}
The map taking $(\sigma, (a,b))\in \mathbb{Z}_2\times (\mathbb{R}^+)^2$ to the 3-form
    \begin{equation*}
        \varphi=A^3e^1\wedge e^{23} + B^3\sum_{i=1}^3 e^i\wedge\omega_1
    \end{equation*}
where $A=\sigma a$ and $B=\sigma b$, is a diffeomorphism onto the space of $Sp(2)\cdot Sp(1)$-invariant positive 3-forms on the 7-sphere. The metric induced by the $G_2$-structure $\varphi$ is given by
    \begin{equation*}
        g_\varphi = A^2\sum_{i=1}^{3}e^i\otimes e^i + A^{-1}B^3\sum_{i=4}^{7}e^i\otimes e^i.
    \end{equation*}
\end{proposition}
All invariant 3-forms are now coclosed, and the Poisson equation becomes
\begin{equation}\label{Poisson eq for Sp(2)Sp(1)}
\begin{split}
    24A^7B^{-6} + 24A^4B^{-3} &= X,\\
    4A + 24A^4B^{-3} + 4A^{-2}B^3 &= Y.
\end{split}
\end{equation}
As our next theorem shows, this system is always solvable.
\begin{theorem}\label{thm for Sp(2)Sp(1)}
The $G_2$-Laplace operator is an orientation-preserving bijection on $Sp(2)\cdot Sp(1)$-invariant positive 3-forms on the 7-sphere and has exactly two fixed points.
\end{theorem}
\begin{proof}
Observe that the image of a 3-form under the $G_2$-Laplacian is always a positive 3-form with the same orientation; therefore, it is an operator on the space $\mathbb{Z}_2 \times (\mathbb{R}^+)^2$.
In light of the scaling properties of this operator (Lemma \ref{scaling of G2 Laplacian}), it is enough to show that its projectivisation is a bijection.
Consider a map $\Phi:\mathbb{R}^+ \to \mathbb{Z}_2 \times (\mathbb{R}^+)^2\subseteq\mathbb{R}^2$ given by $t\mapsto (1,t,\sqrt[3]{t^4})$. Clearly, this map is a bijection onto every line of positive 3-forms. Using equation~\eqref{Poisson eq for Sp(2)Sp(1)}, we obtain
\begin{equation*}
    \Delta_{\Phi(t)}\Phi(t) = \begin{pNiceArray}{c}
   24(t^{-1}+1) \\
  4t+24+4t^2
\end{pNiceArray}.
\end{equation*}
Taking the quotient in \eqref{Poisson eq for Sp(2)Sp(1)} yields
\begin{equation*}
    \frac{Y}{X} = \frac{4t+24+4t^2}{24(t^{-1}+1)} = \frac{t(t^2+t+6)}{6(t+1)}.
\end{equation*}
Direct computations show that
\begin{align*}
    \bigg(\frac{t(t^2+t+6)}{6(t+1)}\bigg)' = \frac{t(2t+1) (t+1)+(t^2+t+6)}{6(t+1)^2}>0, &&
    \lim_{t\to0}\frac{t(t^2+t+6)}{6(t+1)} = 0, &&
    \lim_{t\to\infty}\frac{t(t^2+t+6)}{6(t+1)}=\infty.
\end{align*}
Altogether, this implies that the $G_2$-Laplacian is a bijection on lines, and hence a bijection on the whole space of positive 3-forms, as we wanted to show.

Since for every pair of fixed points we have a preserved line (see Remark \ref{fixed points}), one needs to solve the equation
\begin{equation*}
    \frac{t(t^2+t+6)}{6(t+1)} = \frac{\sqrt[3]{t^4}}{t}=\sqrt[3]{t}.
\end{equation*}
However, this equation is equivalent to the equation
\begin{equation*}
    t^2(t^2+t+6)^3-6^3(t+1)^3 = t^8 + 3 t^7 + 21 t^6 + 37 t^5 + 126 t^4 - 108 t^3 - 432 t^2 - 648 t - 216 = 0.
\end{equation*}
Since there is only one sign change, we can use the Descartes' rule of signs to deduce that there is only one positive root.
\end{proof}
\begin{cor}
    For every positive $Sp(2)\cdot Sp(1)$-invariant 3-form $\mu$ there is a unique positive $Sp(2)\cdot Sp(1)$-invariant solution $\varphi$ to the Poisson equation
    \begin{equation*}
        \Delta_\varphi\varphi=\mu.
    \end{equation*}
\end{cor}

For the case with $Sp(1)\cdot U(1)$ symmetry we have the following result.

\begin{theorem}\label{thm for Sp(2)U(1)}
    There is a neighbourhood around the set of $Sp(2)\cdot Sp(1)$-invariant positive 3-forms in the space of $Sp(2)\cdot U(1)$-invariant positive 3-forms on the 7-sphere where the $G_2$-Laplacian is a bijection onto its image.
\end{theorem}

\begin{proof}
    It is enough to show that on the subspace defined by $B=R$ and $\alpha=0$ the Jacobi matrix of the $G_2$-Laplace operator is non-singular and use the inverse function theorem. Because of Theorem \ref{thm for Sp(2)Sp(1)}, the Jacobi matrix has rank at least 2. Furthermore, when we consider the direction along $\alpha$, we obtain
    \begin{align*}
        \frac{\partial \Delta_\varphi(\varphi)}{\partial\alpha}\bigg|_{B=R,\alpha=0} = (0,0,0, -8A^{-2}B^3-16A^4B^{-3}),
    \end{align*}
    which is linearly independent from $\frac{\partial\Delta_\varphi\varphi}{\partial A}$ and $\left(\frac{\partial}{\partial B}+\frac{\partial}{\partial R}\right)\Delta_\varphi\varphi$ along $B=R$ and $\alpha=0$ since both of these have the last coordinate zero. Next, we find
    \begin{align*}
        &\frac{\partial \Delta_\varphi(\varphi)}{\partial B}\bigg|_{B=R,\alpha=0} \\
        &=(-48A^7B^{-7}-24A^4B^{-4}, -16AB^{-1}+20A^{-2}B^2+72A^4B^{-4}, 8AB^{-1}-72A^4B^{-4}-4A^{-2}B^2, 0).
    \end{align*}
    We will show that the middle two coordinates are never the same, which will imply linear independence from $\frac{\partial\Delta_\varphi\varphi}{\partial A}$, $\left(\frac{\partial}{\partial B}+\frac{\partial}{\partial R}\right)\Delta_\varphi\varphi$ and $\frac{\partial\Delta_\varphi\varphi}{\partial \alpha}$ since all of these have the middle two coordinates the same.
    It follows that the Jacobi matrix has full rank.
    Consider the parameter $s=AB^{-1}\neq 0$. The equality of the middle coordinates of $\frac{\partial\Delta_\varphi\varphi}{\partial B}$ is equivalent to the equation
    \begin{equation*}
        -16s + 20s^{-2} +72s^4 = 8s -72s^4 - 4s^{-2},
    \end{equation*}
    which is equivalent to
    \begin{equation*}
        (12s^3-1)^2 + 23 = 0.
    \end{equation*}
    However, this equation clearly has no solutions. This completes the proof.
\end{proof}

\begin{theorem}\label{thm:fixed points for Sp(2)Sp(1)}
    The $G_2$-Laplacian restricted on the space of $Sp(2)\cdot Sp(1)$-invariant positive 3-forms has eight fixed points that correspond to the nearly parallel $G_ 2$-structures and a continuous family of fixed points that are not nearly parallel $G_2$-structures.
\end{theorem}

\begin{proof}
    Consider $Q=R, \alpha=\beta, A=X$ and $B=Y$ in the system \eqref{Poisson eq for Sp(2)U(1)}. From the fourth equation,
    \begin{equation*}
        R\sin(\alpha) = -8A^{-2}B^2R\sin(\alpha) - 16A^4B^{-4}R\sin(\alpha),
    \end{equation*}
    one concludes that either $R\sin(\alpha)=0$ or $-16Z^2-8Z^{-1}-1 = 0$, where $Z=\frac{A^2}{B^2}$. However, the latter equation has a real solution only for negative $Z$, which clearly never happens. Therefore, $R\sin(\alpha)=0$, and since $R>0$, we obtain that $R\cos(\alpha) = \pm R$.
    We will consider it as a new variable, i.e. ${R'=R\cos(\alpha) \in \mathbb{R}\backslash\{0\}}$.
    Altogether, this implies that every fixed point must be a coclosed form. Combined with Remark \ref{fixed points}, we are thus solving the eigenvalue problem for the operator given in equation \eqref{G2-Laplacian on the coclosed part}, that is,
    \begin{equation*}
        \left(
    \begin{pNiceArray}{ccc}
  A^5B^{-2}R'^{-4} & 0 & 0 \\
  0 & A^{-1}B^4R'^{-4} & 0\\
  0 & 0 & A^{-1}B^{-2}R'^2
\end{pNiceArray}
\begin{pNiceArray}{ccc}
  0 & -4 & -8 \\
  -2 & 2 & -4\\
  -2 & -2 & 0   
\end{pNiceArray}\right)^2 \begin{pNiceArray}{c}
    A\\
    B\\
    R'
\end{pNiceArray} = \lambda\begin{pNiceArray}{c}
    A\\
    B\\
    R'
\end{pNiceArray},
    \end{equation*}
    where $A,B,R'\in\mathbb{R}\backslash\{0\}$ and $A$ and $B$ have the same sign. Let us denote the matrix on the left-hand side by $\textbf{D}^2$. The operator $\textbf{D}=*_\varphi d$ is self-adjoint for the corresponding metric and hence is diagonalizable.
    Furthermore, since we are looking for the eigenvectors and not the eigenvalues, we can multiply both sides of the equation by $A^2B^4R'^8$ to obtain the equivalent problem
    \begin{equation*}
        \left(
    \begin{pNiceArray}{ccc}
  A^6 & 0 & 0 \\
  0 & B^6 & 0\\
  0 & 0 & A^6
    \end{pNiceArray}
    \begin{pNiceArray}{ccc}
  0 & -4 & -8 \\
  -2 & 2 & -4\\
  -2 & -2 & 0   
    \end{pNiceArray}\right)^2 \begin{pNiceArray}{c}
    A\\
    B\\
    R'
    \end{pNiceArray} = \lambda\begin{pNiceArray}{c}
    A\\
    B\\
    R'
    \end{pNiceArray}.
    \end{equation*}
    By an abuse of notation we are going to call the matrix being squared on the left-hand side $\textbf{D}$ as well.
    Note that every eigenvector of a matrix is also an eigenvector of the square of the matrix with the eigenvalue squared.
    We will distinguish two cases: either the matrix of the operator $\textbf{D}^2$(for fixed $A,B,R'$) has the same eigenvectors as the matrix $\textbf{D}$ (for the same $A,B,R'$), or it has strictly more of them. The first case then corresponds to finding a nearly parallel $G_2$-structure. Indeed, by definition, a nearly parallel $G_2$-structure is one satisfying $d\varphi=\tau_0 *\varphi$, where $\tau_0$ is some function. \vspace{5pt}\\
    \textbf{Case 1.} The eigenvectors of $\textbf{D}^2$ are the same as those of $\textbf{D}$. Therefore, we can solve the eigenvalue problem for the operator \textbf{D}, that is,
    \begin{equation*}
    \begin{pNiceArray}{ccc}
  A^6 & 0 & 0 \\
  0 & B^6 & 0\\
  0 & 0 & R'^6
\end{pNiceArray}
\begin{pNiceArray}{ccc}
  0 & -4 & -8 \\
  -2 & 2 & -4\\
  -2 & -2 & 0   
\end{pNiceArray} \begin{pNiceArray}{c}
    A\\
    B\\
    R'
\end{pNiceArray} = \lambda\begin{pNiceArray}{c}
    A\\
    B\\
    R'
\end{pNiceArray}.
    \end{equation*}
    Because the problem is scaling-invariant, we can assume $R'=1$ to obtain the equations
    \begin{align*}
        A^5(2B+4) = -\lambda/2, &&
        B^5(A-B+2) = -\lambda/2, &&
        A+B = -\lambda/2.
    \end{align*}
    The second and third equation can be rewritten as
    \begin{equation*}
        A(B^5-1) = B^6-2B^5+B.
    \end{equation*}
    For $B=1$ the first equation yields $A^5=\frac{A+1}{6}$, and that has only one positive real solution by the Descartes' rule of signs (we do not care about the negative ones since $B$ is positive). This produces two fixed points (cf. Remark \ref{fixed points}).
    Assuming $B\ne 1$, we have
    \begin{equation*}
        A = \frac{B^6-2B^5+B}{B^5-1},
    \end{equation*}
    and substituting into the equation $A^5(2B+4)=A+B$ we end up with
    \begin{equation*}
        \frac{(B^6-2B^5+B)^5(2B+4)}{(B^5-1)^5}=\frac{2B^6-2B^5}{B^5-1},
    \end{equation*}
    which is equivalent to
    \begin{equation*}
        (B^5-2B^4+1)^5(B+2)-(B^5-1)^4(B-1) = 0.
    \end{equation*}
    This polynomial has 4 real roots (one of which is 1).\vspace{5pt}\\
    \textbf{Case 2.} There are eigenvectors of $\textbf{D}^2$ that are not eigenvectors of \textbf{D}. Since both $\textbf{D}^2$ and $\textbf{D}$ are diagonalisable, this happens if and only if there is an eigenvalue $\lambda^2$ of $\textbf{D}^2$ with bigger multiplicity than the multiplicity of the eigenvalue $\lambda$ of $\textbf{D}$.
    This is the situation if and only  if there is a non-zero eigenvalue $\lambda$ of $\textbf{D}$ such that the negative $-\lambda$ is also an eigenvalue of $\textbf{D}$. Since we are dealing with $3\times3$ matrices, it is equivalent to the fact that the trace is an eigenvalue, which in turn is equivalent to
    \begin{align*}
    0 &= \det\begin{pNiceArray}{ccc}
        2B^6 & 4A^6 & 8A^6 \\
        2B^6 & 0 & 4B^6\\
        2R'^6 & 2R'^6 & 2B^6   
    \end{pNiceArray} = -16A^6B^{12} + 64 A^6B^6R'^6 - 16B^{12}R'^6.
    \end{align*}
    And so, we find that $4A^6\ge B^6$ and 
    \begin{equation*}
    R'^6 = \frac{A^6B^6}{4A^6-B^6}.
    \end{equation*}
    One can check that, with this equality the eigenvalues for the operator $\textbf{D}^2$ are
    \begin{equation*}
        \lambda_1 = 4B^6,\: \lambda_2=\lambda_3 = \frac{48A^{12}B^{12}}{4A^6-B^6},
    \end{equation*}
    with corresponding eigenvectors
    \begin{align*}
        v_1=\begin{pNiceArray}{c} -2 \\ \frac{B^6-2A^6}{A^6} \\ 1 \end{pNiceArray}, &&
        v_2=\begin{pNiceArray}{c} \frac{8A^6-2B^6}{B^6} \\ 0 \\ 1 \end{pNiceArray}, &&
        v_3=\begin{pNiceArray}{c} \frac{B^6-2A^6}{B^6} \\ 1 \\ 0 \end{pNiceArray}.
    \end{align*}
    Since the eigenvector must be $(A,B,R')^T$, we obtain the condition
    \begin{equation*}
        A=B\frac{B^6-2A^6}{B^6} + R'\frac{8A^6-2B^6}{B^6}.
    \end{equation*}
    Again, choosing $R'=1$ we obtain an implicit curve
    \begin{equation*}
        AB^6 -B^7+2A^6B-8A^6+2B^6=0,
    \end{equation*}
    where every solution satisfying $4A^6\ge B^6$ is an eigenvector for the $G_2$-Laplacian. When one solves for $B=1$, the condition is satisfied, and so the solutions form a non-empty continuous family. This completes the proof.
\end{proof}

\vspace{2mm}

Observe that for the symmetry $Sp(2)\cdot U(1)$, we cannot get the same results as in Theorems \ref{thm for Spin(7)}, \ref{thm for SU(4)} and \ref{thm for Sp(2)Sp(1)} because the image of the $G_2$-Laplacian is not contained in the space of positive 3-forms. Indeed, consider the Poisson equation \eqref{Poisson eq for Sp(2)U(1)} for any fixed $A, R$ and $\alpha$ with $B$ as a free parameter. Near infinity, we have
\begin{align*}
    X \sim -B^5, && Y \sim B^{11}.
\end{align*}
Therefore, for sufficiently large $B$ the first two coordinates of the image (i.e., $X$ and $Y$) have different signs, and hence the image of the $G_2$-Laplacian is not a positive 3-form.
However, near 0 we have
\begin{align*}
    X \sim B^{-4}, && Y \sim B^2.
\end{align*}
So, by continuity, there exists a positive 3-form, the image of which is a degenerate form.\\

\section{Numerics for $Sp(2)\cdot U(1)$-symmetry}

We have used the Python library \texttt{Matplotlib} to plot the image of the $G_2$-Laplacian in the case of $Sp(2)\cdot U(1)$ symmetry. This provides an illustration for the observations made at the end of Subsection~\ref{Sp(2)U(1) subsection} concerning the non-positivity of the image of the $G_2$-Laplacian. We have run a separate numerical experiment to demonstrate that this $G_2$-Laplacian is not injective. In this section, we describe the details of our experiments.

Because the operator respects scaling (see Lemma \ref{scaling of G2 Laplacian}), the injectivity on the space $\Lambda^3(\mathfrak{m})^{Sp(1)\cdot U(1)}$ is equivalent to the injectivity on the projectivisation $\mathbb{P}\Lambda^3(\mathfrak{m})^{Sp(1)\cdot U(1)}$. 
In order to establish non-injectivity on the projectivisation, we have substituted vectors of the form $(A,B,R,\alpha)=(1,B,R,0)$ and $(A,B,R,\alpha)=(1,B,R,\pi)$ into the system~\eqref{Poisson eq for Sp(2)U(1)}. These vectors form a plane in the space of invariant 3-forms parametrised as in Proposition~\ref{positive 3-forms for Sp(2)U(1)}. Normalising the output $(X,Y,Q\cos\beta,Q\sin\beta)$ to achieve $X=1$, we ensure that this output lies in the same plane. We have plotted the image of rays going from the origin in this plane. The result is shown in Figure 2.

\begin{figure}[h!]
\centering
    \includegraphics[width=15cm]{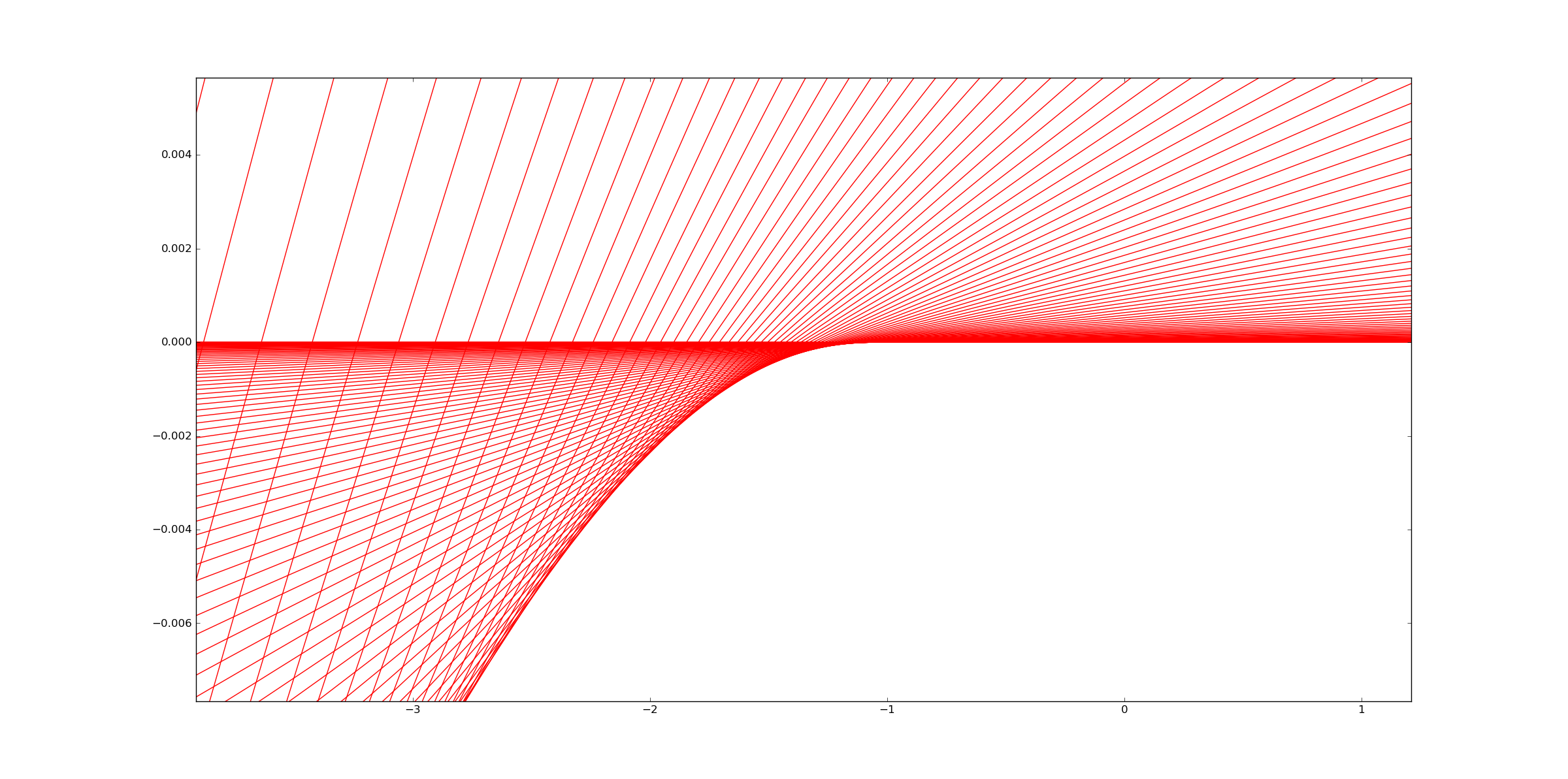}
\caption{Image of the $Sp(2)\cdot U(1)$-invariant rays from  the origin in the plane given by $A=1,\alpha\in\{0,\pi\}$ under the $G_2$-Laplacian}
\end{figure}

The intersections we observe are evidence of the $G_2$-Laplacian not being injective in general. Also, note that everything to the left from the $y$-axis ($x=0$) is a non-positive form.

Similarly, we have also plotted the image in the full projective space, that is, without restrictions on~$\alpha$, to investigate the surjectivity. We have plotted images of progressively bigger clouds of points as well as various clusters. The result can be seen in Figure 3. The top half of the space represents positive forms, and the lower half non-positive forms (i.e., $A$ and $B$ have different signs).

\begin{figure}[htb]
\centering
    \includegraphics[width=.3\textwidth]{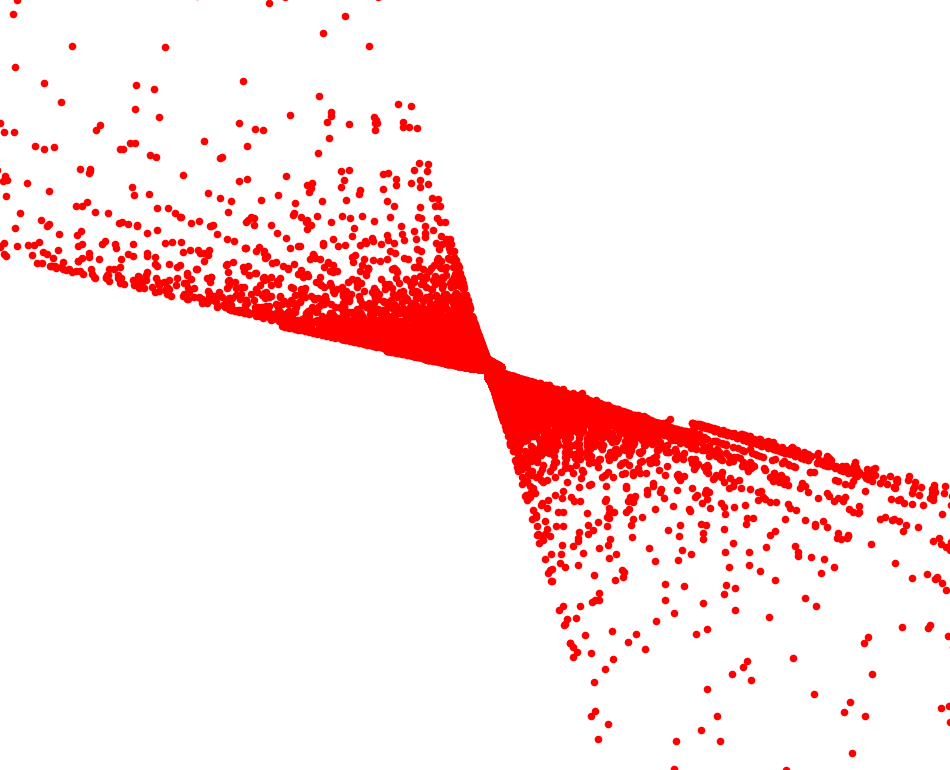}
    \includegraphics[width=.3\textwidth]{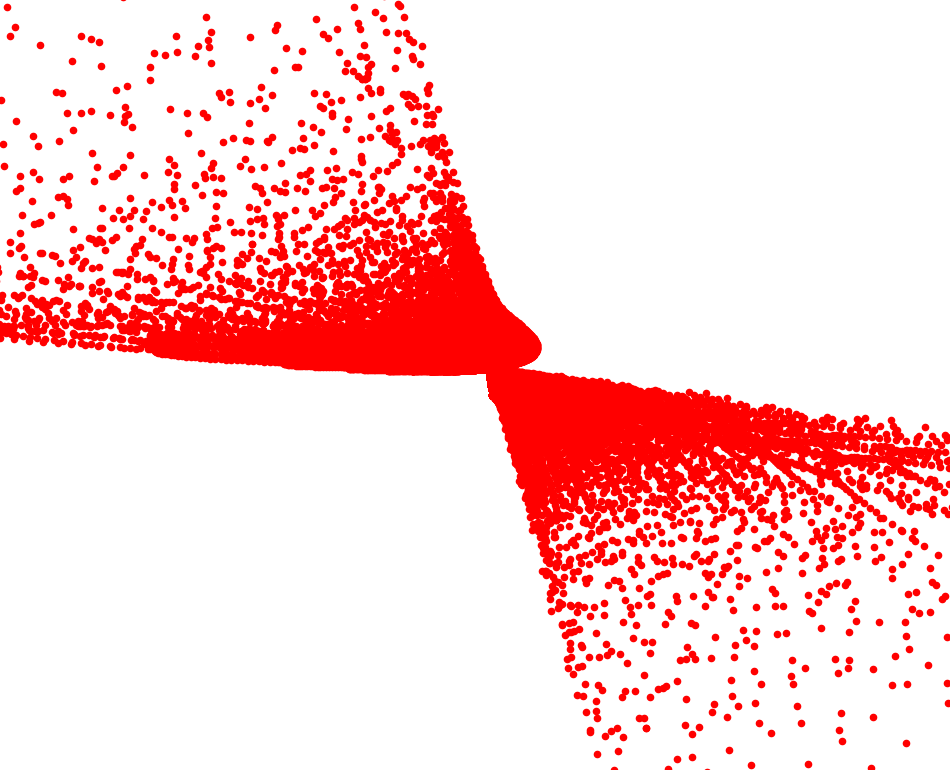}
    \includegraphics[width=.3\textwidth]{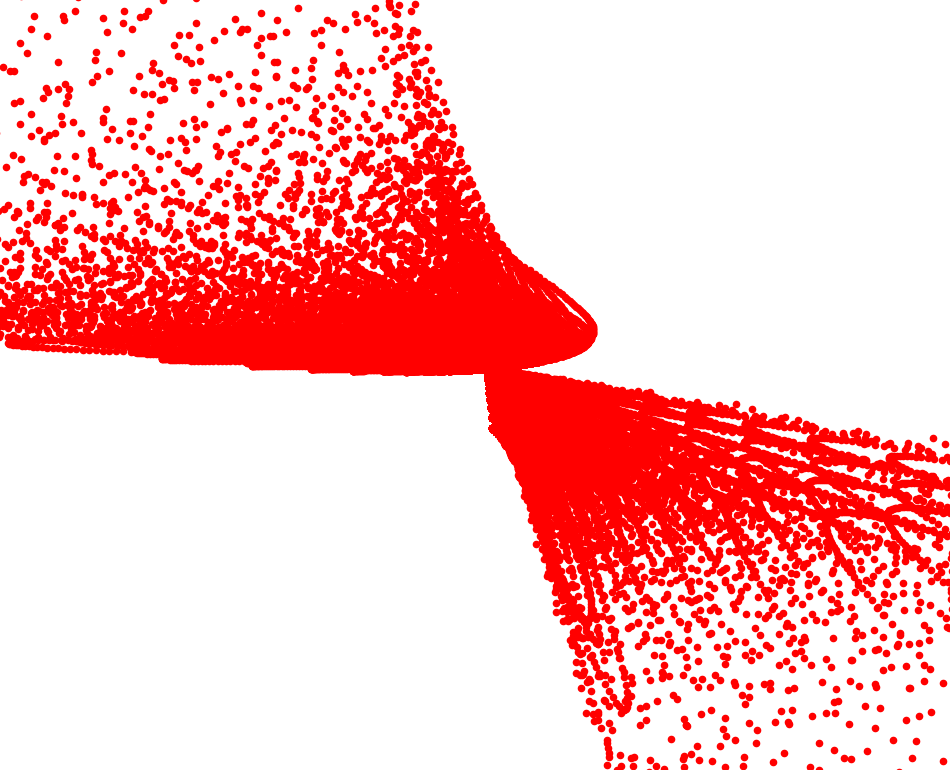}
\caption{Progressively bigger point clouds of the $(Sp(2)\cdot U(1))$-invariant forms mapped by the $G_2$-Laplacian}
\end{figure}

\begin{figure}[h!]
\centering
    \includegraphics[width=.3\textwidth]{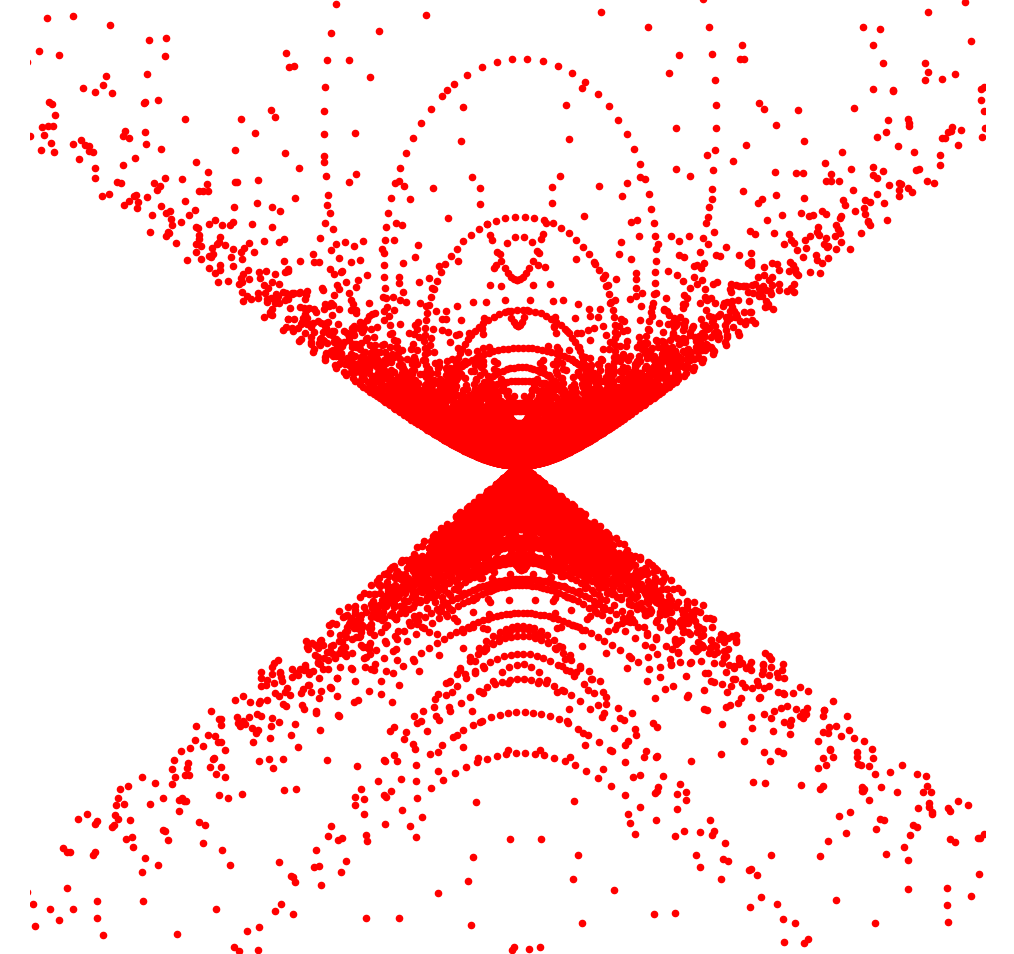}
    \caption{Point cloud of the $(Sp(2)\cdot U(1))$-invariant forms mapped by the $G_2$-Laplacian from the 'side'}
\end{figure}

The top ``cone" appears to be filling out the entire upper half-space as the cloud grows. However, the lower ``cone" is filling only the right side of the lower half-space. Furthermore, there seem to be lines that bound the image in the non-positive forms (see Figure 4 for a different perspective). This suggests that the $G_2$-Laplacian is not surjective onto the space of all $Sp(2)\cdot U(1)$-invariant 3-forms.

\bibliography{References}

\end{document}